\documentclass{amsart}

\usepackage[T1]{fontenc} 
\usepackage[latin1]{inputenc}
\usepackage[spanish,english]{babel}

\usepackage{amsmath}
\usepackage{amssymb}
\usepackage{mathtools}
\usepackage[all]{xy}

\usepackage{amsthm}
\usepackage{caption}
\usepackage{graphicx}
\usepackage{multicol}
\usepackage{a4, amsfonts, amssymb, amsthm, latexsym}
\usepackage{enumerate}
\usepackage{fancyhdr}
\usepackage{color}
\usepackage{nicefrac}
\usepackage{multicol}
\usepackage{ textcomp }
\usepackage{ bbold }
\numberwithin{equation}{section}
\usepackage{mathrsfs,
     epsfig, multicol,anysize,graphicx,enumitem,mdwlist}

\newtheorem{teo}{Teorema}[section]
\newtheorem{lem}[teo]{Lema}
\newtheorem{prop}[teo]{Proposici\'on}

\newtheorem{ex}[teo]{Ejemplo}

\newtheorem{obs}[teo]{Observación}

\newtheorem*{thm*}{Teorema}
\newtheorem*{prop*}{Proposici\'on}

\def\im{\mathop{\mbox{\normalfont im}}}

\def\cok{\mathop{\mbox{\normalfont coker}}}

\DeclareMathOperator{\Hom}{Hom}
\DeclareMathOperator{\md}{mod}

\newtheorem*{defi*}{Definition}

\DeclareMathOperator{\ad}{ad}

\DeclareMathOperator{\g}{\mathfrak{g}}

\newcommand{\mc}[1]{\mathcal{#1}} 
\newcommand{\mb}[1]{\mathbb{#1}} 
\newcommand{\mf}[1]{\mathfrak{#1}} 
\newcommand{\fc}[3]{{#1}: {#2} \rightarrow {#3}}

\marginsize{3cm}{3cm}{3cm}{3cm}

\def\im{\mathop{\mbox{\normalfont im}}}

\def\cok{\mathop{\mbox{\normalfont coker}}}

\title{Introducci\'on a los $\mc{D}$-módulos}
\author{Juan Camilo Arias y Camilo Rengifo}
\address{Departamento de Matem\'{a}ticas, Universidad de los Andes. Bogotá, Colombia}
\email{jc.arias147@uniandes.edu.co}
\address{Facultad de Ingenier\'ia, Universidad de la Sabana, Ch\'ia, Colombia}
\email{camilo.rengifo@unisabana.edu.co }

\begin{document}
\maketitle

{\selectlanguage{spanish}
\begin{abstract}
    Estas notas son las memorias del cursillo dictado en el XXII Congreso Colombiano de Matemáticas en la Universidad del Cauca en Popayán - Colombia. El objetivo de este escrito es brindar un acercamiento a la teoría de módulos sobre el anillo de operadores diferenciales de una variedad algebraica suave. 
    
\end{abstract}}

\begin{abstract}
These are the lecture notes of a short course given at the XXII Colombian Congress of Mathematics held at Universidad del Cauca in Popayán - Colombia. The aim of this paper is to provide an introduction to the theory of modules over rings of differential operators over a smooth algebraic variety.%
    
\end{abstract}

\section{Introducción}

La teoría de los $\mathcal{D}$-m\'odulos tiene su origen como parte del análisis algebraico\footnote{El término análisis algebraico es una palabra atribuida a M. Sato quien buscó estudiar propiedades de funciones y distribuciones analizando los operadores diferenciales lineales que se anulan en estos objetos mediante la teoría de haces, \cite{KKK}.} de la escuela japonesa de Kyoto, liderada por M. Sato y M. Kashiwara entre otros \cite{S1}, \cite{S2}, \cite{S3}, \cite{K}. Uno de los principales objetivos fue el estudio de soluciones de sistemas de ecuaciones diferenciales utilizando herramientas como la teoría de anillos, el álgebra homológica y la teor\'ia de haces. Paralelamente, el matemático ruso-israel\'i J. Bernstein introdujo los $\mathcal{D}$-m\'odulos en los artículos \cite{B2} y \cite{B3} desde el enfoque del análisis complejo. Concretamente, J. Bernstein consideró un polinomio $P$ en $n$ variables complejas y mostró que la función $\mathcal{P}(s)=|P|^{s}$, para $Re(s)>0$, extiende a una función meromorfa de $s$ sobre todo el plano complejo, y tal que toma valores en distribuciones de $\mathbb{C}^n$.  Algunos de los resultados que se desprenden del estudio de los $\mathcal{D}$-m\'odulos son las hiper-funciones de Sato, el análisis microlocal, aplicaciones a la geometría algebraica, la teoría de representaciones y la física matemática.\\

Los $\mathcal{D}$-módulos han mostrado ser de gran utilidad ya que mediantes el uso de $\mathcal{D}$-m\'odulos (holonómicos regulares) se resuelve el problema $21$ de Hilbert, como consecuencia de la correspondencia de Riemann-Hilbert. Adicionalmente, la teoría de los $\mathcal{D}$-m\'odulos provee el contexto apropiado para resolver las conjeturas de Kazhdan-Lusztig.\\

A grosso modo, la correspondencia de Riemann-Hilbert responde a lo siguiente. La monodromía asociada a un sistema lineal de ecuaciones diferenciales da pié a una representación del grupo fundamental del espacio base. Ahora bien, si se tiene una representación del grupo fundamental del espacio base, ¿es posible encontrar un sistema de ecuaciones diferenciales tal que su monodromía coincida con la representación fijada previamente?\\ 

Por otro lado, otra de las grandes aplicaciones de la teoría de los $\mathcal{D}$-m\'odulos es el teorema de localizaci\'on de Beilinson y Bernstein, el cual da una equivalencia entre la categor\'ia de representaciones de un álgebra de Lie semisimple y los $\mathcal{D}$-módulos holonómicos y regulares. \\

Como aplicaci\'on del teorema de localización y de la correspondencia de Riemann - Hilbert se pueden demostrar las conjeturas de Kazhdan-Lusztig, las cuales buscan las f\'ormulas de caracteres para las representaciones irreducibles (finito o infinito dimensionales) de un álgebra de Lie semisimple.\\


A continuación, se ilustra brevemente cómo el uso de los $\mathcal{D}$-módulos permite encontrar el espacio de solución a un sistema de ecuaciones diferenciales desde el punto de vista algebraico. Las nociones que se usan se definirán más adelante, \cite{RTT}.\\

Sea $\mathbb{C}$ el campo de los números complejos. Considere un operador diferencial $P=\sum_{\alpha} g_{\alpha}\partial^{\alpha}$, donde $g_{\alpha}\in \mathbb{C}[X_1, \ldots, X_n]$, $\alpha\in \mathbb{N}^{n}$ es un multi-\'indice $\alpha=(\alpha_1, \ldots, \alpha_n)$ y $\partial^{\alpha}=\partial_1^{\alpha_1}\cdots\partial_n^{\alpha_n}$. Estamos interesados en resolver el problema $P(f)=0$, para $f$ una función polinomial. \\

En general, podemos considerar un sistema  $$ \sum_{j=1}^q P_{ij}(f_j) = 0 $$ con $i=1, \ldots p$ y $P_{ij}$ operadores diferenciales. Este sistema lo podemos escribir de manera matricial como $[P_{ij}][f_j]^T=0$. Para solucionar este sistema de ecuaciones diferenciales se debe encontrar un módulo sobre el anillo de operadores diferenciales. Concretamente, el anillo de operadores diferenciales $D_n$ asociado al anillo de polinomios $\mathbb{C}[X_1,\ldots, X_n]$, está generado por las variables $X_1, \ldots, X_n$ y las derivadas parciales $\partial_1, \ldots, \partial_n$, sujetas a las siguientes relaciones $X_iX_j=X_jX_i$, $\partial_i\partial_j=\partial_j\partial_i$ y la regla de Leibniz. Es decir, 

$$D_n := k[X_1, \ldots, X_n, \partial_1, \ldots, \partial_n]/\langle X_iX_j - X_jX_i , \partial_i\partial_j - \partial_j\partial_i , \partial_iX_i - X_i\partial_i - 1, i,j=1, \ldots, n \rangle $$
\vspace{0.3cm}

Ahora bien, la matriz $[P_{ij}]$ define una aplicación $D_n$-lineal $\cdot[P_{ij}]:D_n^p \rightarrow D_n^q $ definada por $(u_1, \ldots, u_p)\mapsto (\sum_i u_iP_{i1}, \ldots, \sum_i u_iP_{iq})$, cuya imagen es el ideal $J = \langle u_1(\sum_{j=1}^{q}P_{1j}) + \cdots +   u_p(\sum_{j=1}^{q}P_{pj}): u_1, \ldots, u_p\in D_n \rangle$. De lo cual se tiene la secuencia exacta

\begin{equation} \label{coherente} \xymatrix{ D_n^p \ar[r] & D_n^q \ar[r] & D_n^q/J \ar[r] & 0} \end{equation}

El módulo $D_n^q/J$ representa el sistema de ecuaciones diferenciales en cuestión, \cite{S3}. Explícitamente, si $f_1, \ldots, f_q$  son soluciones del sistema, se define una aplicación $\mathbb{C}$-lineal $\sigma : D_n^q \rightarrow \mathbb{C}[X_1, \ldots, X_n]$, $e_i\mapsto f_i$, donde $\{e_i\}$ es la base canónica de $D_n^q$. Note que para $Q\in D_n^q$, se tiene que $\sigma(Q)=Q^T(f_1, \ldots, f_q)$. Así pues, $\sigma(J)=0$ luego $\sigma$ induce una aplicación en el cociente $\overline{\sigma}: D_n^q/J \rightarrow \mathbb{C}[X_1, \ldots, X_n]$. A su vez, si $\tau:D_n^q/J \rightarrow \mathbb{C}[X_1, \ldots, X_n]$ es una aplicación $\mathbb{C}$-lineal, existe una aplicación $\hat{\tau}: D_n^q \rightarrow \mathbb{C}[X_1, \ldots, X_n]$ $\mathbb{C}$-lineal tal que $\hat{\tau}(J)=0$, es decir $\hat{\tau}$ induce la aplicación $\tau$ en el cociente. Defina $g_i=\hat{\tau}(e_i)$, luego para $Q\in J$ se tiene que $Q^T(g_1, \ldots, g_q) = \hat{\tau}(Q)=0$, es decir, $(g_1, \ldots, g_q)$ es solución del sistema en cuestión. Entonces, las soluciones al sistema de ecuaciones diferenciales corresponden al espacio vectorial $\Hom_{\mathbb{C}}(D_n^q/J, \mathbb{C}[X_1, \ldots, X_q])$.\\

Dicho de otra forma, si a la secuencia \eqref{coherente} se le aplica el funtor $\Hom_\mathbb{C}( \_  ,\mathbb{C}[X_1, \ldots, X_n])$, se obtiene la secuencia exacta

\begin{equation}  \xymatrix{ 0 \ar[r] & \Hom_\mathbb{C}( D_n^q/J ,\mathbb{C}[X_1, \ldots, X_n]) \ar[r] & \Hom_\mathbb{C}( D_n^q  ,\mathbb{C}[X_1, \ldots, X_n]) \ar[r]^{[P_{ij}]\cdot} &  \Hom_\mathbb{C}( D_n^p  ,\mathbb{C}[X_1, \ldots, X_n]) }\end{equation}

donde $[P_{ij}]\cdot(\sigma) = [P_{ij}][f_i]^T$ para $\sigma\in\Hom_\mathbb{C}(D_n^q,\mathbb{C}[X_1, \ldots, X_n]$ dada por $\sigma(e_i)= f_i$. Note que, 
$$\Hom_\mathbb{C}(D_n^q/J ,\mathbb{C}[X_1, \ldots, X_n]) = \ker([P_{ij}]\cdot)$$ corresponde precisamente al espacio de soluciones del sistema de ecuaciones diferenciales dado.\\

Por lo tanto, para un sistema de ecuaciones diferenciales $P$ existe un módulo $M_P:=D_n^m/J$, para algunos $n$ y $m$ números naturales y $J$ el ideal generado por $P$, de tal forma que el espacio de soluciones de $P$ corresponde al espacio vectorial $\text{Sol}(M_p):=\Hom_\mathbb{C}(M_P,\mathbb{C}[X_1, \ldots, X_n])$. Adicionalmente, con este método algebraico se pueden obtener soluciones generales al sistema $P$.  Por ejemplo, si $\mathbb{C}[X_1, \ldots, X_n]$ se reemplaza por el anillo de funciones suaves de un abierto $U$ de $\mathbb{R}^n$ se obtendría soluciones suaves de $P$. \\

Nótese que el módulo $M_P$ puede tener diferentes representaciones explícitas. Es decir, pueden existir operadores $P$ y $P'$ tales que $M_P\cong M_{P'}$ como $D_n$-módulos. Por ejemplo, (ver \cite{kaBook} pp. xiii), en una variable $x$, considere el operador diferencial $P=x\frac{d}{dx} - \lambda$ y la ecuación diferencial $P(f)=0$, donde $f$ es una función con valores complejos. Si tomamos $g=xf$, la ecuación diferencial se puede reescribir vía el cálculo
\[
P(g)=P(xf)=xf+x^2\frac{d}{dx}f-\lambda x f,
\]
de lo cual se tiene que $P(g)-g=(P-1)(g)=0$.
Recíprocamente, si consideramos la ecuación $(P-1)g=0$, y $\lambda\neq -1$ obtenemos que $f=\frac{1}{\lambda+1}\frac{dg}{dx}$ satisface $P(f)=0$,
\[
P\left(\frac{1}{\lambda-1}\frac{dg}{dx}\right)=\frac{1}{\lambda+1}\left(\frac{d}{dx}\left(x\frac{dg}{dx}\right)-\frac{dg}{dx}-\lambda\frac{dg}{dx}\right)=\frac{1}{\lambda+1}\frac{d}{dx}\left(x\frac{dg}{dx}-\lambda g-g\right)=0.
\]

Más aún, si $g=xf$ con $P(f)=0$ entonces 

$$\frac{1}{\lambda+1} \frac{dg}{dx} - f = \frac{1}{\lambda+1}\left(\frac{d}{dx}(xf)\right) - f= \frac{1}{\lambda+1}\left(f+x\frac{df}{dx}\right)-\frac{\lambda+1}{\lambda+1}f=\frac{1}{\lambda+1}\left(x\frac{df}{dx}-\lambda f\right)=0.$$

Análogamente si $f=\frac{1}{\lambda+1}\frac{d}{dx}g$ con $(P-1)g=0$, tenemos que $g=xf$.\\

Por lo tanto vemos que las ecuaciones $P(f)=0$ y $(P-1)(g)=0$ son equivalentes. En términos de $D_n$-módulos esto significa que $M_P\cong M_{P-1}$, y por lo tanto el espacio de soluciones a esta (y a cualquiera equivalente a ella) ecuación diferencial está dado por $\Hom_{\mathbb{C}}(M_P, \mathbb{C}[X]) \cong \Hom_{\mathbb{C}}(M_{P-1}, \mathbb{C}[X])$.\\

El art\'iculo est\'a dividido en cinco secciones. En la primera, los $\mathcal{D}$-m\'odulos se estudian desde el punto de vista local, es decir, dado un anillo conmutativo $R$, estudiaremos los m\'odulos sobre el anillo de operadores diferenciales de $R$. Las herramientas necesarias para entender conceptos como variedad caracter\'istica y m\'odulos holon\'omicos se presentarán para ofrecer una explicaci\'on autocontenida. La segunda secci\'on recapitula los conceptos necesarios de la teor\'ia de haces y variedades algebraicas para entender el enfoque global de la teor\'ia de $\mathcal{D}$-módulos. En la tercera sección se presentarán los conceptos de haces y variedades algebraicas indispensables para el estudio global de los $\mathcal{D}$-módulos. En la cuarta sección se buscar\'a introducir muy brevemente los $\mathcal{D}$-m\'odulos globalmente. Se dar\'a la definici\'on de $\mathcal{D}$-m\'odulo para una variedad algebraica suave. Adicionalmente, se mostrar\'a c\'omo conectar estructuras de $\mathcal{D}$-m\'odulo con conexiones planas de un fibrado vectorial y se presentar\'a una lista de ejemplos para ilustrar los conceptos introducidos. Finalmente, en la quinta secci\'on, se presentan algunas aplicaciones a teor\'ia de representaciones y se dar\'a una versi\'on ligera de la correspondencia de Riemann-Hilbert. Es importante decir que estas notas no contienen ning\'un resultado original, todo el material que se presenta corresponde a la forma en que los autores decidieron dar el cursillo.

\section{Nociones preliminares}

 En esta primera sección se definirá el anillo de operadores diferenciales para $\mathbb{C}[X_1,X_2,...,X_n]$ y se presentarán algunas de sus propiedades. Adicionalmente, se darán las definiciones de álgebra de Lie y de módulos sobre un álgebra de Lie. Los resultados de esta sección se pueden encontrar en \cite{DrM} y \cite{H1}. \\

\subsection{Generalidades} Sea $A$ una $\mathbb{C}$-\'algebra conmutativa. Se denota por $\text{End}_{\mathbb{C}}(A)$ la $\mb{C}$-álgebra asociativa de endomorfismos $\mathbb{C}$-lineales de $A$ y por $\text{End}_{A}(A)$ el espacio vectorial de endomorfismos $A$-lineales de $A$. Se sigue que $\text{End}_{A}(A)$ es una sub\'algebra de $\text{End}_{\mathbb{C}}(A)$. El lector puede verificar que $A\cong \text{End}_{A}(A)$ como $\mathbb{C}$-álgebras via el isomorfismo  $a\in A\mapsto \hat{a}:A\to A$, donde por definici\'on $\hat{a}(b)=ab$ es la multiplicaci\'on en $A$. Por tanto, el producto en $A$ coincide con la composición en $\text{End}_{A}(A)$. De ahora en adelante, los elementos de $A$ se identificarán con los endomorfismos de $A$ que son $A$-lineales, i.e. $A\subset\text{End}_{\mathbb{C}}(A)$.\\

 El subespacio vectorial de las derivaciones de $A$ est\'a definido por 
\[
\text{Der}_{\mathbb{C}}(A):=\{T\in \text{End}_{\mathbb{C}}(A) |\, T(ab)=aT(b)+bT(a)\}.
\]

 Para $\phi,\psi\in\text{End}_{\mathbb{C}}(A)$, el conmutador $[\,,\,]$ de $\text{End}_{\mathbb{C}}(A)$ est\'a definido por
$[\phi,\psi]:=\phi\circ\psi-\psi\circ\phi$. Por definición se tiene que el conmutador $[\,,\,]:\text{End}_{\mathbb{C}}(A)\times\text{End}_{\mathbb{C}}(A)\to\text{End}_{\mathbb{C}}(A)$ es una aplicación $\mathbb{C}$-bilineal, antisimétrica y cumple la \emph{identidad de Jacobi},
$$[\phi,[\psi,\rho]]=[[\phi,\psi],\rho]+[\psi,[\phi,\rho]].$$ En otras palabras, $\text{End}_{\mathbb{C}}(A)$ es un \'algebra de Lie, (ver secci\'on \ref{Liesec}). Se deja como ejercicio verificar que el conmutador se anula sobre $\text{End}_{A}(A)\cong A$. 

\subsection{Operadores diferenciales sobre A}\label{defop} Sean $T\in\text{End}_{\mathbb{C}}(A)$ y $p\in\mathbb{Z}_{\geq 0}$. Se dice que $T$ es un operador diferencial de orden menor o igual a $p$ si para todo $a_0,a_1,...,a_p\in A$ se tiene que 
\[
[[...[\,[T,a_0],a_1],...,\,],a_p]=0.
\]
Si $T$ tiene orden menor o igual a cero, para todo $a, a_0\in A$, $(T\circ a_0-a_0\circ T)(a)=T(a_0a)-a_0T(a)=0$, es decir $T$ es un endomorfismo $A$-lineal.

\begin{obs} Por definición si $T$ es un operador diferencial de orden menor o igual a $p$, el operador $[T,a]$ tiene grado menor o igual a $p-1$ para todo $a\in A$. \end{obs}

\begin{obs}
El orden de un operador diferencial se denotará por $ord()$. Si $T$ es un operador diferencial de orden menor o igual a $p$, entonces $ord(T)\leq p$.
\end{obs}

 El subespacio vectorial de todos los operadores diferenciales sobre $A$ se denota por $\text{Diff}_{\mathbb{C}}(A)$.

\begin{lem}
Suponga que $T,S\,\in\text{Diff}_{\mathbb{C}}(A)$ son de orden menor o igual que $n$ y $m$ respectivamente. Entonces, $T\circ S$ es un operador diferencial de orden menor o igual que $n+m$.
\end{lem}

\begin{proof} Para cada $a\in A$ se tiene que $$[T\circ S,a]=T\circ[S,a]+[T,a]\circ S.$$ Nótese que $[S,a]$ es de orden menor o igual a $m-1$ y $[T,a]$ es de orden menor o igual que $n-1$, luego el resultado se sigue por inducción en el grado de la composición.
\end{proof}

 Bajo composición $\circ$ se tiene que $\text{Diff}_{\mathbb{C}}(A)$ es una subálgebra asociativa de $\text{End}_{\mathbb{C}}(A)$. Adicionalmente, $\text{Diff}_{\mathbb{C}}(A)$ tiene estructura de $\mathbb{Z}$-\'algebra. Más aún éste es un anillo filtrado no conmutativo, cuya filtración viene dada por el orden de los operadores diferenciales. Explícitamente, la filtración por el orden está definida así
\[
F_{n}\text{Diff}_{\mathbb{C}}(A)=\{0\},\;\text{si}\,n<0.
\]
\[
F_{n}\text{Diff}_{\mathbb{C}}(A)=\{T\in\text{Diff}_{\mathbb{C}}(A)|\,\text{orden de}\, T\leq n\}.
\]
La filtración por el orden es creciente, i.e., $F_{n}\text{Diff}_{\mathbb{C}}(A)\subset F_{n+1}\text{Diff}_{\mathbb{C}}(A)$. Por otro lado, se tiene que
la filtración por el orden es compatible con la estructura de álgebra de $\text{Diff}_{\mathbb{C}}(A)$, es decir, para todo par de enteros $n,m$
$$F_{n}\text{Diff}_{\mathbb{C}}(A)\circ F_{m}\text{Diff}_{\mathbb{C}}(A)\subset F_{n+m}\text{Diff}_{\mathbb{C}}(A).$$

\begin{lem}\label{propiedades-filtracion} La filtración por el orden cumple las siguientes propiedades
\begin{enumerate}
    \item $F_{0}\text{Diff}_{\mathbb{C}}(A) = A$.
    \item $F_{1}\text{Diff}_{\mathbb{C}}(A) = A\oplus\text{Der}_{\mathbb{C}}(A)$.
    \item $[F_{n}\text{Diff}_{\mathbb{C}}(A),F_{m}\text{Diff}_{\mathbb{C}}(A)]\subset F_{n+m-1}\text{Diff}_{\mathbb{C}}(A)$.
\end{enumerate}
\end{lem}
\begin{proof} 
\begin{enumerate}
    \item Se sigue de $\text{End}_{A}(A)\cong A$.
    \item $\text{Der}_{\mb{C}}(A)\subset F_{1}\text{Diff}_{\mathbb{C}}(A)$ ya que para $T\in\text{Der}_{\mb{C}}(A)$ y $a,b\in A$ se sigue que $$[T,a](b)=T(a(b))-a(T(b))=T(a)b,$$ es decir el conmutador es igual a multiplicar por un elemento de $A$. Nótese que toda $T\in\text{Der}_{\mb{C}}(A)$ cumple que $T(1)=0$, de lo cual se tiene que $T(a)=[T,a](1)$. Ahora bien, si $S\in F_{1}\text{Diff}_{\mathbb{C}}(A)$ defina $T=S-S(1)$ y se sigue que $T(a)=[T,a](1)$ y $T(ab)=[T,ab](1)=bT(a)+aT(b)$. Luego $S=T+S(1)$ donde $T$ es una derivación y $S(1)\in A$. La suma es directa ya que $Der_{\mb{C}}(A)\cap A = \{0\}$.
    \item Por inducción en $n+m$. En el caso base no hay nada que demostrar. El paso inductivo se sigue de la siguiente igual (identidad de Jacobi) 
    $$[[T,S],a]=[[T,a],S]+[T,[S,a]].$$
\end{enumerate}
\end{proof}

 El anillo graduado asociado a la filtración del orden en $\text{Diff}_{\mathbb{C}}(A)$ está dado por 
$$\text{Gr}\text{Diff}_{\mathbb{C}}(A)=\bigoplus_{n\in\mathbb{Z}}\left(F_{n}\text{Diff}_{\mathbb{C}}(A)/F_{n-1}\text{Diff}_{\mathbb{C}}(A)\right)=\bigoplus_{n\in\mathbb{Z}}\text{Gr}_n(\text{Diff}_{\mathbb{C}}(A)).$$

 De la condición $(3)$ del Lema \ref{propiedades-filtracion} se tiene que $\text{Gr}\text{Diff}_{\mathbb{C}}(A)$ es un $A$-\'algebra conmutativa.

\begin{ex}\label{baseder}
Considere $A=\mb{C}[X_1, \ldots, X_n]$ el anillo de polinomios en $n$ variables con coeficientes en $\mb{C}$. Se denota por $D(n)=\text{Diff}_{\mathbb{C}}(\mb{C}[X_1, \ldots, X_n])$ el anillo de operadores diferenciales. Def\'inase los símbolos $\partial_1,\ldots,\partial_n$ mediante la regla $\partial_i(X_j)=\delta_{ij}$ los cuales se extienden a $D(n)$ usando la regla de Leibniz. Se puede mostrar que el conjunto $(\partial_i)_{i=1}^{n}$ forma una base para $\text{Der}_{\mb{C}}(\mb{C}[X_1,\ldots,X_n])$ como $\mb{C}[X_1,\ldots,X_n]$-m\'odulo.

 Ahora bien, considere dos conjuntos de índices finitos $I=\{i_1,\ldots,i_n\}$ y $J=\{j_1,\ldots,j_n\}$,
$$X^{I}=X_1^{i_1}\dots X_n^{i_n},\;\;\partial^{J}=\partial_1^{j_1}\dots \partial_n^{j_n},$$ denotan los productos finitos entre los operadores $X_i$ y $\partial_j$, $1\leq i,j\leq n$. Por definici\'on se tiene que $X^{I}\partial^{J}\in D(n)$ y se puede mostrar que $(X^{I}\partial^{J})_{I,J}$ forman una base como espacio vectorial sobre $\mb{C}$. Es claro que el orden  de $X^{I}\partial^{J}$ es menor o igual a $|J|=\sum_{l=1}^{n}j_l$. En general, para $T=\sum_{|I|\leq p}P_{I}(X_1,\ldots,X_n)\partial^{I}$ donde $P_{I}(X_1,\ldots,X_n)\in\mb{C}[X_1, \ldots, X_n]$, su orden es menor o igual a $p$. 

\end{ex}

\subsection{El s\'imbolo de un operador diferencial}\label{simbolo} Sean $n\geq 1$, $T\in F_n\text{Diff}_{\mathbb{C}}(A)$ y $a_1,...,a_n\in A$. Considere la aplicación $\sigma_n(T):A^{n}\to \text{Diff}_{\mathbb{C}}(A)$ definida por $$\sigma_n(T)(a_1,...,a_n)=[...[[T,a_1],a_2],...,a_n].$$ Como el orden de $T$ es menor o igual que $n$, la aplicación $\sigma_n(T)$ toma valores en $F_{0}\text{Diff}_{\mathbb{C}}(A)\cong A$. 

\begin{lem}
La aplicación $\sigma_n(T)$ es $\mathbb{C}$-multilineal simétrica.
\end{lem}
\begin{proof}
Nótese que para todo $S\in\text{Diff}_{\mathbb{C}}(A)$ y $a,b\in A$ se tiene que $$[[S,a],b]=[[S,b],a],$$ de lo cual se sigue que
$$\sigma_n(T)(a_1,\ldots,a_i,a_{i+1},\ldots,a_n)=\sigma_n(T)(a_1,\ldots,a_{i+1},a_{i},\ldots,a_n),$$ donde $(a_1,\ldots,a_n)\in A^n$.
\end{proof}

\begin{obs}\label{simbolo y orden}
El orden de $T$ es menor o igual a $n-1$ $\iff$ $\sigma_n(T)=0$.
\end{obs}

 Para el caso $A=\mb{C}[X_1,\ldots,X_n]$, el anillo graduado asociado $GrD(n)$ tiene una descripción explícita en t\'erminos de generadores y relaciones. Para presentar tal descripción, se construirá un isomorfismo de $\mb{C}$-\'algebras denotado por $\text{Sym}$. La función $\text{Sym}$ viene inducida por un polinomio $\text{Sym}_p(\,)$ que se define a continuaci\'on. \\

 Sea $T\in D(n)$ un operador diferencial de orden menor o igual a $p$.  
\begin{itemize}
    \item Si $p<0$, entonces $T=0$ y $\text{Sym}_{p}(T)=0$.
    \item Si $p=0$, entonces $T\in\mb{C}[X_1,\ldots,X_n]$ y  $\text{Sym}_0(T)=T$.
    \item Si $p\geq 1$, entonces $\text{Sym}_p(T)$ se construye de la siguiente forma. Para cada $n$-tupla $(\xi_1,\ldots,\xi_n)\in\mb{C}^{n}$, sea $l_{\xi}=\sum_{i=1}^{n}\xi_i X_i\in\mb{C}[X_1,\ldots,X_n]$. Ahora bien, considere la aplicaci\'on $\mathbb{C}^n\to\mathbb{C}[X_1,\ldots,X_n]=F_0D(n)$ dada por,
    $$(\xi_1,\ldots,\xi_n)\mapsto \frac{1}{p!}\sigma_{p}(T)(l_{\xi},\ldots,l_{\xi}),$$ 
    esta se corresponde a un polinomio en $\mb{C}[X_1,\ldots,X_n,\xi_1,\ldots,\xi_n]$. Definimos $\text{Sym}_p(T)$ como el polinomio asociado a la aplicación anterior. 
\end{itemize}

 El polinomio $\text{Sym}_p(T)$ se denomina el \emph{$p$-s\'imbolo} de $T$. 

\begin{ex}{\label{Ejemplo-Sobrey}} Expl\'icitamente, para orden menor igual a $p$ se tiene que,
\begin{enumerate} 
    \item $\text{Sym}_0(X_i)=X_i$.
    \item $\text{Sym}_{1}(\partial_i)=\xi_i$, ya que para $f\in\mb{C}[X_1,\ldots,X_n]$, se tiene  
    \begin{align*}
        [\partial_i,\xi_j X_j](f) & = \xi_j(\partial_iX_j-X_j\partial_i)(f)\\
        & = \xi_j(\partial_i(X_jf)-X_i\partial_j(f))=\xi_j\delta_{ij}f.
    \end{align*}
    \item De los ejemplos anteriores se tiene que si $T=\sum_{|I|\leq p}P_I(X)\partial^{I}$, entonces
    \[
    \text{Sym}_p(T)=\sum_{|I|\leq p}P_I(X)\xi^{I}.
    \]
\end{enumerate}
\end{ex}

\begin{lem}\label{factoriza}
Si $\text{ord}(T)<p$ entonces $\text{Sym}_p(T)=0$
\end{lem}
\begin{proof}
Se sigue de la Observación \ref{simbolo y orden}.
\end{proof}

 Por el Lema \ref{factoriza} la aplicación $\text{Sym}_p$ factoriza por el anillo asociado graduado a la filtración dada por el orden,
$$\text{Sym}_p:\text{Gr}_p D(n)\to \mb{C}[X_1,\ldots,X_i,\xi_1,\ldots, \xi_n],$$ la cual induce la aplicación
$$\text{Sym}:\text{Gr} D(n)\to \mb{C}[X_1,\ldots,X_i,\xi_1,\ldots, \xi_n].$$

\begin{lem}\label{Sym-map-algebras}
Sean $T$ y $S$ dos operadores diferenciales de orden $p$ y $q$ respectivamente. Entonces $$\text{Sym}_{p+q}(T\circ S)=\text{Sym}_p(T)\text{Sym}_q(S)$$
\end{lem}
\begin{proof}
Ver pp. 21 de \cite{DrM}.
\end{proof}

 En otras palabras del Lema \ref{Sym-map-algebras} se sigue que $\text{Sym}$ es una aplicación de $\mb{C}$-álgebras. Adicionalmente, por el Ejemplo \ref{Ejemplo-Sobrey}, se tiene que $\text{Sym}$ es sobreyectiva.

\begin{lem}{\label{Sym-iny}}
$ \text{Sym}_{p}(T)=0\iff \text{ord}(T)\leq n-1.$
\end{lem}
\begin{proof}
Por inducción en $p$. Ver pp. 22 de \cite{DrM}. 
\end{proof}

 Del lema \ref{Sym-iny} se tiene que la aplicación $\text{Sym}$ es inyectiva.

\begin{teo}\label{anillo-graduado}
La aplicación $\text{Sym}$ es un isomorfismo de $\mb{C}$-\'algebras.
\end{teo}

\subsection{M\'odulos, soporte y variedad caracter\'istica}\label{sopycharvar} Debido a que $D(n)$ es un anillo no conmutativo, a priori hay que distinguir entre la categoría de $D(n)$-módulos a izquierda y a derecha. Para efectos de las definiciones que se darán en esta sección, se trabajará con $D(n)$-módulos a izquierda.\\

 Por un $D(n)$-módulo $M$ se entender\'a un grupo abeliano dotado de la estructura de $\mb{C}[X_1,\ldots,X_n]$-módulo en el cual los s\'imbolos $\partial_i$ act\'uan como derivaciones en $\text{End}_{\mb{C}}(M)$. En esta secci\'on se introducir\'a un tipo especial de $D(n)$-m\'odulos. Para esto se necesita recordar algunas nociones de teor\'ia de anillos. \\ 

Para cada $x\in\mb{C}^{n}$ sea $\mathfrak{m}_x=\{f\in\mb{C}[X_1,\ldots,X_n]|f(x)=0\}$ el ideal maximal de $\mb{C}[X_1,\ldots,X_n]$ que consiste de funciones que se anulan en $x$. La localización de $\mb{C}[X_1,\ldots,X_n]$ en $\mathfrak{m}_x$ se denotará por  $\mb{C}[X_1,\ldots,X_n]_x$. El ideal maximal de $\mb{C}[X_1,\ldots,X_n]_x$ será denotado por $\mathfrak{n}_x$ y ya que $\mb{C}[X_1,\ldots,X_n]_x$ es un anillo local regular tenemos que $\mathfrak{n}_x=(\mathfrak{m}_x)_x$.\\

 Consideremos a $M$ con su estructura de $\mb{C}[X_1,\ldots,X_n]$-módulo a izquierda. Para cada $x\in \mb{C}^n$ se denota por $M_x$ la localización de $M$ en $\mathfrak{m}_x$, es decir
$$M_x=M_{\mathfrak{m}_x}=(\mb{C}[X_1,\ldots,X_n]-\mathfrak{m}_x)^{-1}M\cong \mb{C}[X_1,\ldots,X_n]_{x}\otimes_{\mb{C}[X_1,\ldots,X_n]} M.$$
Luego,el soporte de $M$ est\'a dado por $$\text{Supp}(M)=\{x\in\mb{C}^n|M_x\neq 0\}.$$

 Para un $I$ un ideal de $\mb{C}[X_1, \ldots, X_n]$ denotaremos por $V(I)$ a la variedad de ceros del ideal $I$, es decir, $V(I):=\{ (x_1, \ldots, x_n) \in \mb{C}^n | f(x_1,\ldots, x_n)=0, \mbox{ para toda } f\in I\}$.\\ 

 Luego el soporte del $D(n)$-módulo $M$ está contenido en $\mb{C}^n$. Más aún, como $M$ es un $\mb{C}[X_1,\ldots,X_n]$-módulo se puede mostrar que $$\text{Supp}(M)=V(\text{Ann}_{\mb{C}[X_1,\ldots,X_n]}(M))$$ donde $\text{Ann}_{\mb{C}[X_1,\ldots,X_n]}M=\{f\in\mb{C}[X_1,\ldots,X_n]\,|\,fm=0,\,\forall m\in M\}$. En particular, si $M$ es finitamente generado, entonces $\text{Supp}(M)$ es un subconjunto cerrado de Zariski en $\mb{C}^n$.\\

 A cada $T\in D(n)$ la aplicación $Sym$, ver secci\'on \ref{simbolo}, produce un polinomio en $\mb{C}^{2n}$ denotado por $Sym(T)$. Si $M$ es $D(n)$-m\'odulo generado por un ideal $I=\langle T_1,\ldots, T_k\rangle \subset D(n)$, entonces tiene sentido considerar la variedad $V(\{Sym(T) | T\in I\})$, la cual llamamos variedad característica. Esta variedad busca describir el espacio de soluciones del sistema $T_1 u=\ldots=T_k u=0$ donde $u\in\mb{C}[X_1,\ldots, X_n]$. En muchos casos ocurre que la variedad característica coincide con $V(Sym(T_1), \ldots, Sym(T_k))$, pero puede ocurrir que ésta sea más pequeña. Los $D(n)$-módulos para lo que la variedad es de dimensión mínima son de particular interés, ver sección \ref{definiciones-dmodulos}. El estudio de estos módulos se encuentra en \cite{RTT} pp. 81. \\

 En general, para un $D(n)$-módulo $M$, definimos su variedad característica de la siguiente manera. Supongamos que $M$ admite una \emph{buena} filtración (ver pp. 10 de \cite{DrM} para la definición) $FM$ que es compatible con la filtración dada por el orden $FD(n)$ en $D(n)$. En particular, si $M$ es un $D(n)$-módulo finitamente generado, entonces existe una \emph{buena} filtración $FM$, (ver pp.10 de \cite{DrM}).\\

 El ideal característico de $M$ asociado a una \emph{buena} filtración $FM$  se define por el radical del anulador de $Gr(M)$ con respecto a $GrD(n)$.
$$J(M)=\sqrt{\text{Ann}_{\text{Gr(D(n))}}(Gr(M))}.$$
\begin{prop}
$J(M)$ no depende de la \emph{buena} filtración de $M$. \\

\begin{proof} Ver pp. 15 de \cite{B}. \end{proof}
\end{prop}
 La \emph{variedad característica} de $M$ est\'a definida por la variedad cortada por el idela $J(M)$, 
$$\text{Ch}(M)=V(J(M)).$$
Si se cumple que $\text{dim}(\text{Ch}(M))=n$ diremos que $M$ es \emph{holonómico}.

\begin{ex}
$\mb{C}[X_1,\ldots,X_n]$ es un $D(n)$-m\'odulo holonómico.
\end{ex}

\begin{ex}
$M=D(n)/I$ donde $I=\langle T_1,\ldots, T_k\rangle$ ideal de $D(n)$, es un módulo holonómico. 
\end{ex}

\subsection{Álgebras de Lie}\label{Liesec}
Un \'algebra de Lie es un $\mb{C}$-espacio vectorial $\g$ dotado de una aplicaci\'on bilineal $[\;,\;]:\g\times \g\rightarrow \g$, llamada corchete de Lie, que cumple las siguientes condiciones

\begin{itemize}
\item Antisimetría $[X,Y] = -[Y,X]$, para toda $X,Y\in\g$.
\item Identidad de Jacobi $[X,[Y,Z]] = [[X,Y],Z] + [Y,[X,Z]]$ para todo $X,Y,Z\in\g$
\end{itemize}

\begin{ex}
\begin{itemize}

\item El espacio vectorial de las derivaciones $\mb{C}$-lineales, $Der_{\mb{C}}(A)$ para $A$ una $\mb{C}$-álgebra conmutativa, donde el corchete está dado por el conmutador $[\phi,\psi]=\phi\circ\psi - \psi\circ\phi$.
    \item Las matrices $n\times n$ con el corchete dado por el conmutador $[X,Y]=XY-YX$ para dos matrices $X,Y$. Este espacio se denota por $\mf{gl}_n(\mb{C})$ y es llamado \emph{álgebra lineal general}.
    \item El subespacio de $\mf{gl}_n(\mb{C})$ cuyas matrices tienen traza cero con el mismo corchete de $\mf{gl}_n(\mb{C})$ es un álgebra de Lie la cual se denota por $\mf{sl}_n(\mb{C})$ y es llamada \emph{álgebra lineal especial}. En el caso $n=2$, una base para $\mf{sl}_2(\mb{C})$ est\'a formada por las matrices $E=\begin{pmatrix} 0 & 1 \\ 0 & 0 \end{pmatrix}$, $F=\begin{pmatrix} 0 & 0 \\ 1 & 0 \end{pmatrix}$, $H=\begin{pmatrix} 1 & 0 \\ 0 & -1 \end{pmatrix}$. Las reglas de conmutación están dada por $[E,F]=H$, $[H,E]=2E$ and $[H,F]=-2F$.
    \item Sea $V$ un espacio vectorial sobre $\mb{C}$. El espacio vectorial de transformaciones lineales $V\to V$, denotado ${\mf{gl}(V)}$, tiene estructura de álgebra de Lie donde el corchete está dado por $[f,g]=f\circ g- g\circ f$. En particular, si $V$ es de dimensi\'on finita $n$, mediante un isomorfismo $V\cong \mb{C}^n$ se tiene que ${\mf{gl}(V)} \cong \mf{gl}_n(\mb{C})
    $, el cual es un isomorfismo que depende de $V\cong \mb{C}^n$.
\end{itemize}
\end{ex}

 Dadas dos \'algebras de Lie $(\mf{g}_1,[,]_1)$ y $(\mf{g}_2,[,]_2)$ un morfismo $\fc{\varphi}{\mf{g}_1}{\mf{g}_2}$ de álgebras de Lie es una aplicación lineal tal que $\phi([X,Y]_1)=[\phi(X),\phi(Y)]_2$ para todo $X,Y\in\mf{g}$. Se dice que $I$ es un \emph{ideal} de $\mf{g}$ si es un subespacio de $\mf{g}$ tal que $[\g,I]\subseteq I$, es decir, $I$ es cerrado bajo el corchete de $\mf{g}$. El kernel y la imagen de un morfismo de \'algebras de Lie son ejemplos de ideales. Adicionalmente, un \'algebra de Lie $\g$ es \emph{simple} si $[\mf{g},\mf{g}]\neq 0$  y  sus \'unicos ideales son $\mf{g}$ y $0$.\\

 Una representaci\'on de $\g$ es un morfismo de \'algebras de Lie $\fc{\rho}{\g}{\mf{gl}(V)}$. Se denota la acci\'on de $\g$ sobre $V$ como $X\cdot v := \rho(X)(v)$, para todo $X\in \g$ y $v\in V$. Note que el corchete de dos elementos $[X,Y]$ act\'ua sobre $V$ mediante la regla $[X,Y]\cdot v = X\cdot(Y\cdot v) - Y\cdot(X\cdot v)$. En este caso se dice que $V$ es un $\g$-m\'odulo o una representaci\'on de $\g$. Los morfismos entre $\g$-m\'odulos se definen de la manera natural. La categor\'ia de $\g$-m\'odulos se denotar\'a por $\g$-Mod.\\

 El \'algebra envolvente universal de $\g$, denotada por $U(\g)$, se define como:   
$$U(\g) = T(\g)/\langle X\otimes Y - Y\otimes X - [X, Y]\; | \; X,Y\in \g \rangle,$$ donde $T(\g)$ es el álgebra tensorial de $\g$. Es importante notar que $U(\g)$ es una $\mb{C}$-\'algebra asociativa con unidad, noetheriana y sin divisores de cero. La filtraci\'on del \'algebra tensorial $T(\g)$ induce una filtraci\'on en $U(\g)$ de tal forma que $Gr U(\g) \cong S(\g)$, donde $S(\g)$ es el \'algebra sim\'etrica de $\g$; este isomorfismo recibe el nombre de \emph{Teorema PBW} (\cite{H1}, Teorema 17.3). En particular el Teorema PBW implica que el conjunto de palabras de elementos de $\g$ forma un conjunto de generadores para $U(\g)$, (una base PBW).

\begin{ex}
Para $\mf{sl}_2(\mb{C})$, los elementos de la forma $F^iH^jE^k$ forman una base $PBW$, para todos $i,j,k\in \mb{N}$.
\end{ex}

\begin{obs} Cualquier $\g$-m\'odulo es un m\'odulo para $U(\g)$ al extender la acci\'on al producto tensorial. Rec\'iprocamente, todo m\'odulo para $U(\g)$ es un m\'odulo para $\g$ al considerar la acci\'on restricta a los elementos de $\g$. Por lo tanto $U(\g)$-Mod es equivalente a $\g$-Mod.\end{obs}

 Para un álgebra de Lie $\g$ la representación adjunta está dada por la aplicación $\ad : \g \rightarrow \mf{gl}(\g)$, donde para cada $X\in\g$, $\ad_X: \g \rightarrow \g$ es tal que $\ad_X(Y) = [X,Y]$ para toda $Y\in \g$. En. el caso particular de $\mf{sl}_2(\mb{C})$, la representaci\'on adjunta en los elementos b\'asicos viene dada por $\ad_E=\begin{pmatrix} 0 & -2 & 0 \\ 0 & 0 & 1\\ 0 & 0 & 0 \end{pmatrix}$, $\ad_F=\begin{pmatrix} 0 & 0 & 0 \\ -1 & 0 & 0\\ 0 & 2 & 0 \end{pmatrix}$, $\ad_H=\begin{pmatrix} 2 & 0 & 0 \\ 0 & 0 & 0\\ 0 & 0 & -2 \end{pmatrix}$.

\section{Haces y Variedades Algebraicas}

 En esta sección se presentan las definiciones b\'asicas de variedades algebraicas y haces definidos sobre estos espacios. Adicionalmente, se ejemplificarán en algunos casos las definiciones presentadas en esta sección. El contenido de esta sección está inspirado en \cite{H, KS, B, RTT}.

\subsection{Algunas observaciones en la teoría de haces sobre espacios} Sea $X$ un espacio topológico. Denote por $Op(X)$ la categoría de los conjuntos abiertos de $X$, i.e. la categoría cuyos objetos son los conjuntos abiertos de $X$ y para cualesquiera dos conjuntos abiertos $U$ y $V$ de $X$  existe un morfismo de $U$ a $V$ si y sólo si $U\subseteq V$. La categoría opuesta  de $Op(X)$ se denota $Op(X)^{o}$.\\

 Un prehaz\footnote{La palabra (pre)haz se utiliza como traducción de la palabra francesa {\it (pré)faisceaux} o de la palabra inglesa {\it (pre)sheaf}. Dependiendo de la escuela, en español se utiliza la palabra {\it (pre)gavilla} como sinónimo de (pre)haz.} $\mc{F}$ sobre $X$ con valores en una categoría $\mc{C}$ es un functor $\mc{F}:Op(X)^{o} \rightarrow \mc{C}$. Explícitamente para todo $U\subseteq X$ subconjunto abierto, $\mc{F}(U)$ es un objeto de $\mc{C}$. Para cada inclusión $\iota: U \hookrightarrow V$ existe un morfismo ``restricción'' $\mc{F}(\iota)= \rho^{V}_{U} : \mc{F}(V) \rightarrow \mc{F}(U)$. Más aún, para todo abierto $U$, $\rho^{U}_{U}=id_{\mc{F}(U)}$ y para cualquier tripleta de abiertos $U,V,W$ tales que $U\subseteq V \subseteq W$ se tiene que $\rho^{W}_{U} = \rho^{V}_{U}\circ\rho^{W}_{V}$. Si la categoría $\mc{C}$ tiene objeto cero $0$ entonces $\mc{F}(\varnothing) = 0$.\\

\begin{obs} En este artículo la categoría $\mc{C}$ corresponderá a alguna de las siguientes: la categoría de conjuntos {\bf Sets}, la categoría de grupos abelianos {\bf Ab}, la categoría de anillos {\bf Rings} o la categoría de $R$-módulos {\bf $R$-Mod} para un anillo $R$ dependiendo del contexto.\end{obs}

 Para cualquier abierto $U\subseteq X$ los elementos de $\mc{F}(U)$ se llaman secciones locales de $\mc{F}$. Si $\iota: U\subseteq V$ es una inclusión entre dos conjuntos abiertos de $X$ y $s\in \mc{F}(V)$, la restricción $\rho^{V}_{U}(s)$ se denota por $s|_{U}$. Es usual denotar el conjunto de secciones locales $\mc{F}(U)$ por $\Gamma(U, \mc{F})$. El functor $\Gamma(U, \;\_)$ se llama functor de secciones locales. En el caso $U = X$, el functor $\Gamma(X; \;\_)$ es llamado functor de secciones globales de $X$. \\

 En adelante se trabajará bajo el supuesto que $\mc{C}$ es una categoría en la cual existen productos. Un prehaz $\mc{F}$ sobre $X$ con valores en $\mc{C}$ es un \emph{haz}, si para todo abierto $U$ de $X$ y cualquier cubrimiento abierto $\{U_i\}$ de $U$ el siguiente diagrama es un ecualizador en $\mc{C}$:

\[ 
\xymatrix{ \mc{F}(U) \ar[r] & \prod_{i} \mc{F}(U_i) \ar@<0.5ex>[r] \ar@<-0.5ex>[r] & \prod_{i, j} \mc{F}(U_i\cap U_j)  }  
\]

\vspace{0.5cm}

 En otras palabras, para todo abierto $U\subseteq X$ y todo cubrimiento $\{U_i\}$ de $U$ se tiene que $F$ cumple las siguientes propiedades

\begin{enumerate}
    \item Si $s,t\in \mc{F}(U)$ son tales que $s|_{U_i} = t|_{U_i}$ para toda $i$, entonces $s=t$.
    \item Si $s_i\in \mc{F}(U_i)$ es tal que $s_i|_{U_i\cap U_j} = s_j|_{U_i\cap U_j}$ entonces existe (única) $s\in \mc{F}(U)$ tal que $s|_{U_i}=s_i$
\end{enumerate}

\begin{obs} Si existe una sección  ``cero'' en $\mc{F}(U)$, por ejemplo $\mc{C}$ es {\bf Ab}, entonces la condición (1) dice que toda sección que es localmente la sección cero debe ser la sección cero. \end{obs}

\begin{ex} Sea $X$ un espacio topológico, para todo abierto $U\subseteq X$ se define $C^{0}_{X}(U)=\{ f:U\rightarrow \mb{R} |\quad f \mbox{ es continua } \}$. Claramente $C^0_{X}$ es un prehaz y puesto que su naturaleza es puramente local se tiene que $C^{0}_{X}$ es un haz. Si $X$ es una variedad suave, se define $C^{\infty}_{X}(U)=\{ f:U\rightarrow \mb{R} |\quad f \mbox{ es suave } \}$, y se tiene que la asignación $C^{\infty}_{X}$ es un haz. Similarmente, para $X$ una variedad compleja, el haz de funciones holomorfas se define de manera similar y para $X$ una variedad algebraica (o esquema) se tiene el haz de funciones regulares. \end{ex}

\begin{ex} Sea $X$ un espacio topológico y sea $A$ un grupo abeliano con la topología discreta. Defina $\mc{A}(U) = \{ f:U\rightarrow A | \quad f \mbox{ es continua } \}$. Entonces $\mc{A}$ es un haz llamado haz localmente constante. Si $U$ es conexo, $\mc{A}(U)\cong A$, caso contrario es isomorfo una suma directa de copias de $A$. \end{ex}

 Para cualesquiera dos haces $\mc{F}$ y $\mc{G}$ sobre $X$ con valores en $\mc{C}$, un morfismo $\varphi: \mc{F} \rightarrow \mc{G}$ es una transformación natural entre los functores $\mc{F}$ y $\mc{G}$, i.e. para todo par de abiertos $U\subseteq V$ de $X$, $\varphi(U): \mc{F}(U) \rightarrow \mc{G}(U)$ es un morfismo en la categoría $\mc{C}$ para el cual el siguiente diagrama conmuta:

\[ \xymatrix{ \mc{F}(V) \ar[r]^{\varphi(V)} \ar[d]_{\rho^{V}_{U}} & \mc{G}(V) \ar[d]^{\rho^{V}_{U}}\\  \mc{F}(U) \ar[r]^{\varphi(U)} & \mc{G}(U). }  \]

\vspace{0.5cm}

 Se denotará la categoría de prehaces sobre $X$ con valores en $\mc{C}$ por $PSh(X, \mc{C})$; si la categoría $\mc{C}$ es clara en el contexto de trabajo, entonces sólo se escribirá $PSh(X)$, y se denotará su subcategoría plena de haces como $Sh(X)$. Cabe notar que si la categoría $\mc{C}$ es \emph{abeliana}, entonces la categoría $Sh(X)$ también lo es. En lo que sigue la categoría $\mc{C}$ se supondrá ser una categoría abeliana.\\

 Se sigue que todo haz es un prehaz, es decir, existe el functor de olvido $for : Sh(X) \rightarrow PSh(X)$. Sin embargo, no todo prehaz es un haz. Pero es sabido que el functor $for$ cuenta con un functor adjunto a izquierda $i^{+}:PSh(X) \rightarrow Sh(X)$ el cual a cada prehaz $\mc{F}$ le asocia un haz llamado ``sheafification'' $i^{+}\mc{F}$. Al ser functores adjuntos se tiene el siguiente isomorfismo de conjuntos $$ \Hom_{PSh(X)}(\mc{F}, for(\mc{G})) \cong \Hom_{Sh(X)}(i^{+}(\mc{F}), \mc{G}) $$ lo cual es equivalente a decir que para cualquier prehaz $\mc{F}$, y cualquier haz $\mc{G}$ y morfismo de prehaces $\psi:\mc{F} \rightarrow \mc{G} = for(\mc{G})$, (note que $\mc{G} = for(\mc{G})$ como prehaz), existe \'unico morfismo $\varphi: i^{+}(\mc{F}) \rightarrow \mc{G}$ tal que $\psi = \varphi \circ \eta_{\mc{F}} $, donde $\eta_{\mc{F}}: \mc{F}\rightarrow i^{+}(\mc{F})$ es el morfismo de adjunción. Por $\mc{F}^{+}$ se denota el haz $i^{+}(\mc{F})$ el cual admite una descripción explicita dada por

$$  \mc{F}^{+}(U) = \Bigg\{ s:U\rightarrow \bigcup_{x\in U}\mc{F}_x \quad |\quad s(x)\in \mc{F}_x \mbox{ y } \begin{cases} \forall x\in U, \exists V\ni x, V\subseteq U \mbox{ y }  t\in \mc{F}(V)\\ \mbox{ tal que } \forall y\in V \mbox{ se tiene que } t_y=s(y) \end{cases}\Bigg\} \quad  \Bigg\}.  $$

\vspace{1cm}

 Todo morfismo de haces con valores en {\bf Ab}  define los siguientes prehaces asociados a \'el. Suponga que $\varphi:\mc{F} \rightarrow \mc{G}$ es un morfismo de haces con valores en {\bf Ab}.

\begin{itemize}
    \item El kernel de $\varphi$: Para todo abierto $U\subseteq X$, $(\ker\varphi)(U)=\ker(\varphi(U))$.
    \item El cokernel de $\varphi$: Para todo abierto $U\subseteq X$, $(\cok\varphi)(U)=\cok(\varphi(U))$.
    \item La imagen de $\varphi$: Para dodo abierto $U\subseteq X$, $(\im\varphi)(U)=\im(\varphi(U))$.
\end{itemize}

\vspace{0.3cm}

 Se puede mostrar que el kernel de $\varphi$ es de hecho un haz, sin embargo, esto no ocurre con los prehaces imagen y cokernel. Por lo tanto los haces imagen y cokernel de $\varphi$ son las ``sheafificaciones'' de estos prehaces y se denotan $\im\varphi$ y $\cok\varphi$ respectivamente.\\

 Ya que uno de los objetivos de los haces es entender información local, es importante definir g\'ermenes de secciones locales. Sea $\mc{F}$ un haz sobre $X$, el ``stalk'' sobre $x\in X$ está dado por $$\mc{F}_x := \varinjlim_{U\ni x}\mc{F}(U)=\bigsqcup_{x\in U}F(U)/\sim.$$
Así, el stalk $\mc{F}_x$ tiene por elementos clases de equivalencia $[f,U]$ donde $U$ es un abierto y $f\in \mc{F}(U)$. Dos clases $[f,U]$ y $[g,V]$ son equivalentes si existe abierto $W\subseteq U\cap V$  tal que $f|_W = g|_W$. De lo cual se tiene que el germen de una sección $s\in \mc{F}(U)$ se define como la imagen en el stalk $s_x\in \mc{F}_x$. \\

 Dados dos haces $\mc{F}'$ y $\mc{F}$ sobre un espacio topológico $X$, $\mc{F}'$ es un \emph{subhaz} de $\mc{F}$ si para cualquier abierto $U$ de $X$, $\mc{F}'(U)$ es un subobjeto de $\mc{F}(U)$ y el morfismo restricción de $\mc{F}'$ se hereda de $\mc{F}$.\\

 Sean $\mc{F}$, $\mc{G}$ y $\mc{H}$ haces sobre $X$ y sean $\varphi: \mc{F}\rightarrow \mc{G}$ y $\psi:\mc{G}\rightarrow \mc{H}$ morfismos de haces, entonces:

\begin{itemize}
    \item $\varphi$ es un monomorfismo si y sólo si $\ker\varphi=0$, i.e.,  $\varphi_x$ es un monomorfismo en $\mc{C}$. 
    \item $\varphi$ es un epimorfismo si y sólo si $\im\varphi=\mc{G}$, i.e., $\varphi_x$ es un epimorfismo en $\mc{C}$.
    \item $\xymatrix{0\ar[r]&\mc{F}\ar[r]^{\varphi}&\mc{G}\ar[r]^{\psi}&\mc{H}\ar[r]&0}$ es una sucesión exacta corta de haces si y sólo si $\varphi$ es un monomorfismo, $\psi$ es un epimorfismo y $\ker\psi=\im\varphi$, i.e.,  $\xymatrix{0\ar[r]&\mc{F}_x\ar[r]^{\varphi_x}&\mc{G}_x\ar[r]^{\psi_x}&\mc{H}_x\ar[r]&0}$ es una sucesión exacta corta en la categoría $\mc{C}$.
\end{itemize}

\vspace{0.5cm}

\subsection{Variedades algebraicas} La idea intuitiva de una variedad afín corresponde al conjunto de ceros de un sistema de ecuaciones polinomiales. Una variedad algebraica, ser\'a una variedad que admite un cubrimiento finito por variedades afines, es decir, es un espacio que localmente es el conjunto de ceros de un sistema polinomial. Estos conjuntos de ceros no necesariamente coinciden de un abierto a otro. En esta sección se darán algunas ideas básicas sobre variedades algebraicas. Por simplicidad el campo de base ser\'an los números complejos, $\mathbb{C}$.\\

 Sea $\mb{C}[X_1, \ldots, X_n]$ el anillo de polinomios en $n$ variables. Considere un ideal $I$ de $\mb{C}[X_1, \ldots, X_n]$ y defina el conjunto $V(I) = \{x\in \mb{C}^n\; | \; f(x)=0, \; \forall f\in I\}$. $V(I)$ se llama conjunto algebraico. Como primeros ejemplos tenemos que $V(0)=\mb{C}^n$ y que $V(1)=\emptyset$. Claramente, $V(I)\subseteq V(0)=\mb{C}^n$ para todo ideal $I$ del anillo de polinomios.\\

 Los conjuntos algebraicos son cerrados bajo uniones finitas y bajo intersecciones arbitrarias. Recuerde que $V(0)=\mb{C}^n$ y $V(1)=\emptyset$ son conjuntos algebraicos. Así, los conjuntos algebraicos forman una topología (por conjuntos cerrados) de $\mb{C}^n$. Esta topología se llama \emph{la topología de Zariski}. Los abiertos básicos de esta topología se denotan por $D(f) = \mb{C}^n \setminus V(f)$, para $f$ un polinomio.\\

\begin{obs} El espacio $\mb{C}^n$ est\'a dotado de dos topologías, la topología analítica y la topología de Zariski. Ya que el conjunto de ceros de un conjunto de polinomios es cerrado en la topología analítica, tenemos que la topología analítica es más fina que la topología de Zariski.\end{obs}

 El espacio $\mb{C}^n$ con la topología de Zariski, está dotado de manera natural de una familia distinguida de funciones, a saber, las funciones polinomiales. Estas funciones están en correspondencia biyectiva con el anillo de polinomios $\mb{C}[X_1, \ldots, X_n]$. A un polinomio visto como una función sobre $\mb{C}^{n}$ lo llamaremos una función regular, y el conjunto de funciones regulares de $\mb{C}^n$ se denotará por $\mc{O}_{\mb{c}^n}(\mb{C}^n) : = \mb{C}[X_1, \ldots, X_n]$.\\ 

 Las funciones regulares se pueden definir tambi\'en para todo abierto (con la topología de Zariski) $U$ de $\mb{C}^n$, el conjunto de tales funciones se denota por $\mc{O}_{\mb{C}^n}(U)$ y está en correspondencia con las funciones $f: U \rightarrow \mb{C}$ tales que para cada $V\subset U$ abierto básico, $f|_V$ es una función racional que no tiene polos en $V$. Se puede mostrar que la asignación $U\mapsto \mc{O}_{\mb{C}^n}(U)$ para todo abierto de $\mb{C}^n$ es un haz de $\mb{C}$-álgebras. El par $(\mb{C}^n, \mc{O}_{\mb{C}^n})$ se llama \emph{espacio afín $n$-dimensional}.\\

 Sea $Y$ un conjunto cerrado irreducible del espacio afín $\mb{C}^n$, es decir $Y$ no puede ser expresado como la unión de dos subconjuntos cerrados propios. $\mc{O}_Y(Y)$ se define como el conjunto de funciones regulares en $Y$, es decir, $\mc{O}_{Y}(Y) = \mc{O}_{\mb{c}^n}(\mb{C}^n)/I(Y)$, donde $I(Y)=\{f\in \mc{O}_{\mb{c}^n}(\mb{C}^n) \; | \; f(y)=0 \; \forall y\in Y\}$ es el ideal que define a $Y$. Similarmente, se puede definir un haz de $\mb{C}$-álgebras $\mc{O}_Y$ asociado a la variedad $Y$.\\

 Ahora bien, una \emph{variedad algebraica afín} es un par $(Y,\mc{O}_Y)$ donde $Y$ es un cerrado irreducible del espacio afín $\mb{C}^n$ y $\mc{O}_Y$ es el haz de funciones regulares. El par $(Y,\mc{O}_Y)$ se escribe simplemente como $Y$, teniendo presente que este espacio está dotado del haz estructural $\mc{O}_Y$.\\ 

 Dadas dos variedades algebraicas afines $X$ y $Y$, un morfismo $\varphi:X\rightarrow Y$ es una función continua tal que para todo abierto $U$ de $X$ y toda función regular $g\in \mc{O}_Y(U)$ se tiene que $g\circ \varphi\in \mc{O}_X(\varphi^{-1}(U))$. El morfismo $\varphi$ es un isomorfismo si $\varphi$ es un homeomorfismo y la asignación en anillos de funciones regulares $g\mapsto g\circ \varphi$ es un isomorfismo de anillos.\\

 Decimos que un par $(X, \mc{O}_X)$, donde $X$ es un espacio topológico y $\mc{O}_X$ es un haz de $\mb{C}$-álgebras sobre $X$, es una {\it variedad algebraica} si $X$ admite una cubierta finita por variedades algebraicas afines.\\

\begin{obs}
Las variedades algebraicas serán \emph{separadas}, intuitivamente esto significa ser Hausdorff en el sentido usual.
\end{obs}

\begin{ex} La l\'inea proyectiva $\mb{P}^1$ se define como el conjunto $\mb{C}\sqcup\{\infty\}$, donde $\infty$ es un símbolo, el punto al infinito. Equivalentemente, $\mb{P}^1$ se define como el espacio de clases de equivalencia de parejas $(x_0,x_1)$ de elementos de $\mb{C}^2$ tale que $x_0\neq0$ y $x_1\neq 0$, bajo la relación de equivalencia $(x_0,x_1) \sim (\lambda x_0, \lambda x_1)$ con $\lambda\in \mb{C}\setminus\{0\}$. La clase de equivalencia de un punto en $\mb{P}^1$ se denotará por $[x_0:x_1]$. Las funciones regulares en los abiertos de $\mb{P}^1$ son localmente cocientes de polinomios homogéneos en dos variables del mismo grado. Se puede mostrar que $\mb{P}^1$ es una variedad algebraica la cual tiene cubierta $U_0, U_1$ donde $U_0 = \mb{P}^1 - \{[1,0]\} \cong \mb{C}$ y $U_1 = \mb{P}^n - \{[0,1]\} \cong \mb{C}$.
\end{ex}

\begin{obs}
El haz de funciones regulares de $\mb{P}^1$ se puede construir ``pegando'' los haces de su cubierta afín.
\end{obs}

\subsection{$\mc{O}_X$-módulos} Los haces que trabajaremos en el resto de estas notas tienen la estructura extra de módulo. Veremos su definición, algunas propiedades y ejemplos de los mismos.\\

 Suponga que $(X,\mc{O}_X)$ una variedad algebraica y sea $\mc{F}$ un haz sobre $X$. Se dice que $\mc{F}$ es un $\mc{O}_X$-módulo si para todo abierto $U$ de $X$ se tienen que $\mc{F}(U)$ es un $\mc{O}_X(U)$-módulo, y si para $V\subseteq U$, $V$ abierto, el morfismo restricción $\mc{F}(U)\rightarrow \mc{F}(V)$ es compatible con la estructura de módulo vía el morfirmo $\mc{O}_X(U)\rightarrow \mc{O}_X(V)$. Un $\mc{O}_X$-módulo $\mc{I}$ es un haz de ideales si $\mc{I}(U)$ es un ideal de $\mc{O}_X(U)$ para todo abierto $U$ de $X$.\\

 Sean $\mc{F}$ y $\mc{G}$ dos $\mc{O}_X$-módulos. Un morfismo $\varphi:\mc{F}\rightarrow \mc{G}$ de haces es un morfismo de $\mc{O}_X$-módulos si $\varphi(U)$ es un $\mc{O}_X(U)$-homomorfismo. Un $\mc{O}_X$-módulo $\mc{F}$ se dice \emph{libre} si $\mc{F}\cong \mc{O}_X^n$, para $n\in \mb{Z}_{\geq 0}\cup\{\infty\}$. $\mc{F}$ se dice \emph{localmente libre} si existe una cubierta $\{U_i\}$ de $X$ tal que $\mc{F}|_{U_i}$ es libre.\\

 Sea $X$ una variedad algebraica y $\mc{F}$ un $\mc{O}_X$-módulo. Decimos que  $\mc{F}$ es \emph{cuasi-coherente} si admite resolución $\mc{O}_X^{I}\rightarrow \mc{O}_X^{J} \rightarrow \mc{F} \rightarrow 0$, donde $I$ y $J$ son dos conjuntos indexantes arbitrarios. Adem\'as $\mc{F}$ es \emph{coherente} si los conjunto $I$ y $J$ son finitos.\\

\begin{obs}
En otras palabras, la definición anterior es equivalente a la siguiente; un haz $\mc{F}$ es cuasi-coherente si para todo abierto afin $U$ el haz $\mc{F}|_U$ es generado por un $\mc{O}_X(U)$-módulo $M_U$. $\mc{F}$ es coherente si el módulo $M_U$ es finitamente generado. Note que el módulo depende del abierto que se escoja. 
\end{obs} 

 A continuación se presentan dos ejemplos relevantes para el estudio de los $\mc{D}$-módulos, a saber, el haz tangente y el haz de 1-formas diferenciales. Cabe mencionar que el haz tangente y el haz de 1-formas aparecen en geometría diferencial como fibrados vectoriales naturales sobre $X$.\\

\begin{ex}[El haz tangente] Denote por $\underline{End}_{\mb{C}}(\mc{O}_X)$ el haz de endomorfismos lineales del anillo $\mc{O}_X$. Una sección $\theta\in End_{\mb{C}}(\mc{O}_X)(X)$ se dice una derivación si para todo abierto $U$ de $X$ se tiene que $\theta|_U\in Der_{\mb{C}}(\mc{O}_X(U))$, donde $Der_{\mb{C}}(\mc{O}_X(U))$ denota el espacio de derivaciones lineales del anillos $\mc{O}_X(U)$.

 Denote por $\mc{T}_X(U)$ el conjunto de derivaciones sobre $U$. La asignación $U\mapsto \mc{T}_X(U)$ define un $\mc{O}_X$-módulo el cual se denota por $\mc{T}_X$. Dado que los módulos $Der_{\mb{C}}(\mc{O}_X(U))$ son finitamente generados, el haz $\mc{T}_X$ es coherente.
\end{ex}

\begin{ex}
Sea $X=\mb{C}^n$. Entonces $\mc{O}_X(X)=\mb{C}[X_1, \ldots, X_n]$ y como se vió en el ejemplo \ref{baseder}, $Der_{\mb{C}}(\mb{C}[X_1, \ldots, X_n])$ est\'a generado por los s\'imbolos $\partial_1, \ldots, \partial_n$ tales que $\partial_i(X_j)=\delta_{ij}$ y extendidos por la regla de Leibniz. Para un abierto de $X$, las derivaciones sobre tal abierto se corresponderán con las restricciones de las derivaciones de $X$. Si $Y\subseteq X$ una variedad af\'in cerrada, entonces $\mc{O}_Y(Y) = \mb{C}[X_1, \ldots, X_n]/I(Y)$ y $Der_{\mb{C}}(\mc{O}_Y(Y))= \{\theta\in Der_{\mb{C}}(\mb{C}[X_1, \ldots, X_n])\; | \; \theta(I)\subseteq I\}$, es decir, $Der_{\mb{C}}(\mc{O}_Y(Y))$ está generado por los generadores $\partial_i$ tales que $\partial_i(I(Y))\subseteq I(Y)$.
\end{ex}

 Para poder definir el haz de formas diferenciales recuerde que dada una $\mb{C}$-álgebra $A$ entonces el $A$-módulo de 1-formas diferenciales $\Omega_A$ se define como el módulo libre generado en los s\'imbolos $da$ para toda $a\in A$ sujeto a las relaciones $d(a+a')=da+da'$, $d(aa') = (da)a' + a(da)$ para todas $a,a'\in A$ y $dc=0$ para toda $c\in \mb{C}$. En particular, si $A=\mb{C}[X_1, \ldots, X_n]$ se tiene que $\Omega_A$ es el módulo generado por los símbolos $dX_1, \ldots, dX_n$. Más aún se puede mostrar que $dX_i(\partial_j)=\delta_{ij}$, es decir, $\Omega_A$ y $Der_{\mb{C}}(A)$ son duales como espacios vectoriales.

\begin{ex}[El haz de 1-formas] Para $X$ una variedad algebraica y $U$ un abierto afín de $X$, $\Omega_{\mc{O}_X(U)}$ denota el $\mc{O}_X(U)$-módulo de 1-formas diferenciales del anillo $\mc{O}_X(U)$. La asignación $U\mapsto \Omega_{\mc{O}_X(U)} = \Omega_X(U)$ define un haz de $\mc{O}_X$-módulos el cual se denota por $\Omega_X$. Por definición este es un haz coherente.
\end{ex}

 Para cerrar esta sección, se definen el tipo de variedades algebraicas con las cuales se trabajará en el resto de esta nota. Este tipo de variedades permitirán entender de manera mucho más concreta los objetos de interés, los $\mc{D}$-módulos.\\

 Sea $X$ una variedad algebraica y sea $x\in X$. Recuerde que el stalk del haz $\mc{O}_X$ en el punto $x$ se denota por $\mc{O}_{X,x}$. Este es un anillo local con ideal maximal $\mf{m}_x$. A este ideal pertenecen todas (los g\'ermenes de) las funciones que se desvanecen en $x$. Se dice que $X$ es \emph{suave} en el punto $x\in X$ si el anillo $\mc{O}_{X,x}$ es un anillo local regular, es decir, si la dimensión del espacio vectorial $\mf{m}_x/\mf{m}_x^2$ coincide con la dimensión de Krull de $\mc{O}_{X,x}$. Equivalentemente, si $\Omega_{X,x}$ es libre de rango dimensión de Krull de $\mc{O}_{X,x}$. La dimensión de $X$ en el punto $x$ se denota por $\dim_xX$ y es por definición la dimensión de $\mf{m}_x/\mf{m}_x^2$.\\

 Una variedad algebraica es \emph{suave} si lo es en cada punto $x\in X$. A una variedad suave se le asocia un número llamado dimensión y denotado por $\dim X$, que por definición es $\dim_xX$ para cualquier $x\in X$. Como el campo base son los números complejos, entonces las variedades suaves consideradas con la topología analítica tienen estructura de variedades complejas. Ejemplos de variedades algebraicas suaves son el espacio afín y el espacio proyectivo.\\

\begin{teo}\label{teodelabase}
Sea $X$ una variedad algebraica suave de dimensión $n$. Entonces para toda $x\in X$, existe un abierto afín $V\ni x$ y existen funciones regulares $X_i\in \mc{O}_X(V)$ y $\partial_i\in \mc{T}_X(V)$ para $i=1, \ldots, n$ que cumplen $[\partial_i, \partial_j]=0$, $\partial_i(X_j)=\delta_{ij}$ y $\mc{T}_X|_V \cong \bigoplus_{i=1}^n \mc{O}_X|_V \partial_i$. 
\end{teo}

\begin{proof} Ver Teorema A.5.1 en \cite{RTT}. 
\end{proof}

 Al sistema $\{X_1, \ldots, X_n, \partial_1, \ldots, \partial_n\}$ definido en una vecindad afín de un punto $x\in X$ se le llamará \emph{sistema de coordenadas locales} en $x$.

\section{$\mc{D}$-módulos}\label{Dmoddef} 

 En esta sección se presentarán las definiciones del haz de operadores diferenciales sobre una variedad algebraica suave y los $\mc{D}$-módulos sobre tal haz.\\

\subsection{Definiciones}\label{definiciones-dmodulos} Sea $X$ una variedad algebraica suave de dimensión $n$. El haz de operadores diferenciales sobre $X$, denotado por $\mc{D}_X$, se define como el $\mb{C}$-subhaz de $\underline{End}_{\underline{\mb{C}}}(\mc{O}_X)$ generado por $\mc{O}_X$ y $\mc{T}_X$. En otras palabras, para todo abierto $U$ de $X$ el anillo $\mc{D}_X(U)$ está generado por s\'imbolos $\{\hat{f}, \hat{\theta}\; | \; f\in\mc{O}_X(U), \theta\in \mc{T}_X(U) \}$ sujetos a las relaciones:
\vspace{0.2cm}
\begin{itemize}
    \item $\hat{f} + \hat{g} = \widehat{f+g}$ y $\hat{f}\hat{g} = \widehat{fg}$, para $f,g\in \mc{O}_X(U)$.
    \item $\hat{\theta} + \hat{\zeta} = \widehat{\theta + \zeta}$ y $[\hat{\theta},\hat{\zeta}] = \widehat{[\theta,\zeta]}$, para $\theta, \zeta\in \mc{T}_X(U)$.
    \item $\hat{f}\hat{\theta} = \widehat{f\theta}$, $f\in \mc{O}_X(U)$, $\theta\in \mc{T}_X(U)$.
    \item $[\hat{\theta}, \hat{f}] = \widehat{\theta(f)}$, $f\in \mc{O}_X(U)$, $\theta\in \mc{T}_X(U)$.
\end{itemize}
\vspace{0.3cm}
 Equivalentemente, el haz de operadores diferenciales sobre $X$ se puede definir de la siguiente forma. Para cualquier punto $x\in X$ considere $\{X_i, \partial_i\}_{i=1}^n$ un sistema local de coordenadas alrededor de $x$ en $U\ni x$ una vecindad abierta afín. Entonces, 

$$ \mc{D}_X|_U \cong \bigoplus_{\alpha\in \mb{N}^n} \mc{O}_X|_U \partial_1^{\alpha_1}\ldots\partial_n^{\alpha_n}$$
donde $\alpha=(\alpha_1, \ldots, \alpha_n)$.\\

\begin{ex} 
El álgebra de funciones globales para $X=\mb{C}^n$ es el anillo de polinomios en $n$ variables $\mb{C}[X_1, \ldots, X_n]$. Como se vio previamente, $\mc{D}_{\mb{C}^n} = \text{Diff}_{\mb{C}}(\mb{C}[X_1, \ldots, X_n])= D_n$. Expl\'icitamente, $D_n = \mb{C}[X_1, \ldots, X_n, \partial_1, \ldots, \partial_n]/\langle X_iX_j - X_jX_i , \partial_i\partial_j - \partial_j\partial_i , \partial_iX_i - X_i\partial_i - 1, i,j=1, \ldots, n \rangle $. 
\end{ex}

\begin{ex} 
Para $\mb{C}^{*}=\mb{C}\setminus \{ 0 \}$, se tiene que $\Gamma(\mb{C}^{*}, \mc{D}_{\mb{C}^{*}}) \cong \mb{C}[z, z^{-1}, \partial]/\langle [\partial, z] = 1, \; [\partial, z^{-1}] = -z^{-2} \rangle$. 
\end{ex}

\begin{ex}\label{locBBP1} 
Sea $X=\mb{P}^1$ la l\'inea proyectiva. Sean $0=[1:0]$ y $\infty=[0:1]$ y considere la cubierta de $\mb{P}^1$ cuyos abiertos af\'ines son $U_0 = \mb{P}^1\setminus\{ \infty \}$ y $U_1 = \mb{P}^1\setminus\{ 0 \}$. Identifique $U_0$ con $\mb{C}$ mediante el morfismo $z: U_0\rightarrow \mb{C}$, $z([1,:x_1]) \mapsto x_1$ e identifique $U_1$ con $\mb{C}$ mediante el morfismo $\zeta: U_1\rightarrow \mb{C}$, $z([x_0:1]) \mapsto x_0$. Note que $U_0\cap U_1 \cong \mb{C}^{*}$, donde $z$ y $\zeta$ se relacionan mediante la equaci\'on $z= 1/\zeta$. \\

 Se sigue que $\Gamma(U_0, \mc{D}_{\mb{P}^1}) \cong \mb{C}[z, \partial]/ \langle [\partial, z]=1\rangle$, $\Gamma(U_1, \mc{D}_{\mb{P}^1}) \cong \mb{C}[\zeta, \partial_{\zeta}]/ \langle [\partial_{\zeta}, \zeta]=1\rangle$ y $\Gamma(U_0\cap U_1, \mc{D}_{\mb{P}^1}) \cong \mb{C}[z, z^{-1},\partial]/ \langle [\partial, z]=1, [\partial, z^{-1}] = -z^{-2} \rangle$. Los morfirmos de restricción están dados por $z\mapsto z$, $\partial\mapsto \partial$ para el caso $\Gamma(U_0, \mc{D}_{\mb{P}^1}) \rightarrow \Gamma(U_0\cap U_1, \mc{D}_{\mb{P}^1})$ y por $\zeta \mapsto z^{-1}$ y $\partial_{\zeta} \mapsto -z^2\partial$ para el caso $\Gamma(U_1, \mc{D}_{\mb{P}^1}) \rightarrow \Gamma(U_0\cap U_1, \mc{D}_{\mb{P}^1})$.\\

 Ahora bien, como una sección global $s\in \Gamma(\mb{P}^1, \mc{D}_{\mb{P}^1})$ debe satisfacer que la restricción a $U_0$, $s|_{U_0}$, es de la forma $z^i\partial^j$ para algunas $i,j$, y la restricción a $U_1$, $s|_{U_1}$, es de la forma $\zeta^n\partial_{\zeta}^m$ para algunas $n,m$ y que además sobre $U_0\cap U_1$ tales restricciones coincidan. Se sigue que $\Gamma(\mb{P}^1, \mc{D}_{\mb{P}^1}) = span\{ 1, z^i\partial^j\; | i\leq j + 1, \; j>0\}$. Más aún, $\Gamma(\mb{P}^1, \mc{D}_{\mb{P}^1})$ tiene base dada por $\{-\partial, -2z\partial, z^2\partial \}$.\\

 Observe que si $E=-\partial$, $H=-2z\partial$ y $F=z^2\partial$, se tiene que $[H,E]=2E$, $[H,F]=-2F$ y $[E,F]=H$, es decir, $E,H$ y $F$ cumplen las relaciones del álgebra de Lie $\mf{sl}_2(\mb{C})$ y por lo tanto se obtiene un morfismo sobreyectivo $U(\mf{sl}_2(\mb{C})) \rightarrow \Gamma(\mb{P}^1, \mc{D}_{\mb{P}^1})$, cuyo kernel está generado por el elemento de Casimir $C=H^2+2EF+2FE$. Por lo tanto, $U(\mf{sl}_2(\mb{C}))/\langle C \rangle \rightarrow \Gamma(\mb{P}^1, \mc{D}_{\mb{P}^1})$ es un isomorfismo.
\end{ex}
\vspace{0.3cm}
 Sea $U$ un abierto en $X$ y considere un operador $T\in \mc{D}_X(U)$. Se dirá que $T$ es de orden menor o igual a $p$ si para todo abierto af\'in $V\subseteq U$ se tiene que $T|_{V}$ es un operador diferencial de orden menor o igual a $p$, (ver definici\'on \ref{defop}). El orden definido sobre los operadores diferenciales induce una filtraci\'on creciente en $\mc{D}_X(U)$ que se denota por $F\mc{D}_X(U)$. Esta filtración induce una filtraci\'on en $\mc{D}_X$ denotada por $F\mc{D}_X$, de tal forma que el haz graduado $Gr \mc{D}_X$ es un haz de anillos conmutativos tal que $Gr_0 \mc{D}_X \cong \mc{O}_X$.\\

 Por medio de esta filtraci\'on se puede mostrar que $\mc{D}_X$ es un haz cuasi-coherente de $\mc{O}_X$-m\'odulos tanto a izquierda como a derecha. Adem\'as los haces $F_p\mc{D}_X$ son $\mc{O}_X$-coherentes y por lo tanto $Gr_p\mc{D}_X$ es $\mc{O}_X$-coherente. Como ya vimos en el caso de anillos (ver secci\'on 2), la filtraci\'on de grado uno est\'a generada por el anillo y las derivaciones. En este caso se tiene el resultado an\'alogo, es decir, $F_1\mc{D}_X \cong \mc{O}_X\oplus \mc{T}_X$.\\

 Cabe notar que el Teorema \ref{anillo-graduado} se puede globalizar de la siguiente forma. Sea $X$ una variedad algebraica suave, se puede construir una aplicaci\'on s\'imbolo como en el caso de $\mb{C}$-\'algebras, de tal forma que $Gr\mc{D}_X$ es isomorfo como haz graduado al haz de funciones regulares del espacio cotangente (ver \cite{DrM}).\\

 Decimos que un haz $\mc{F}$ de $\mc{O}_X$-módulos es un $\mc{D}_X$-módulo a izquierda, si para todo abierto $U$ de $X$, $\mc{F}(U)$ es un $\mc{D}_X(U)$-módulo a izquierda para el cual las acciones son compatibles con los morfismos de restricción. \\

\begin{obs}
En otras palabras, una estructura de $\mc{D}_X$-m\'odulo sobre un $\mc{O}_X$-m\'odulo $\mc{F}$ se corresponde con un morfismo $\mc{O}_X$-lineal de \'algebras de Lie $\nabla : \mc{T}_X \rightarrow End_{\mb{C}}(\mc{F})$, $\theta \mapsto \nabla_{\theta}$. Es decir, un morfismo que cumple $\nabla_{f\theta} = f\nabla_{\theta}$, $\nabla_{[\theta_1, \theta_2]} = [\nabla_{\theta_1}, \nabla_2]$ y $\nabla_{\theta}(fs) = f\nabla_{\theta}(s) + \theta(f)s$ para toda $f\in \mc{O}_X$, $\theta, \theta_1, \theta_2\in \mc{T}_X$ y $s\in \mc{F}$.
\end{obs}

 En el caso que $\mc{F}$ sea un $\mc{D}_X$-m\'odulo localmente libre sobre $\mc{O}_X$, diremos que el morfismo $\nabla$ es una \emph{conexi\'on plana} para el fibrado vectorial asociado a $\mc{F}$. En general, esto no ocurre con regularidad pues los $\mc{D}_X$-m\'odulos son en su mayor\'ia cuasi coherentes sobre $\mc{O}_X$. Por lo tanto, los $\mc{D}_X$-m\'odulos generalizan la idea de las conexiones planas para cualquier rango.\\

 Un caso particular donde se puede garantizar que un $\mc{D}_X$-m\'odulo sea localmente libre es el siguiente: Si $\mc{F}$ es un $\mc{D}_X$-m\'odulo que adem\'as es $\mc{O}_X$-coherente, entonces $\mc{F}$ es localmente libre sobre $\mc{O}_X$. Por lo tanto, se dice que un $\mc{D}_X$-m\'odulo es una \emph{conexi\'on integrable} o s\'olamente una conexi\'on si es $\mc{O}_X$-coherente.\\

 Una clase de $\mc{D}_X$-m\'odulos de particular importancia son los llamados \emph{holon\'omicos}. Estos $\mc{D}_X$-m\'odulos se pueden definir como se vi\'o en la sección \ref{sopycharvar} en términos de variedades características, pero ya que este tipo de variedades son un poco más complicadas de definir en el caso general, es suficiente decir que los módulos holonómicos son los $\mc{D}_X$-m\'odulos que gen\'ericamente son conexiones, es decir, son m\'odulos para los cuales existe un abierto denso en $X$ tal que su restricci\'on a \'este es una conexi\'on. Los m\'odulos holon\'omicos forman una categor\'ia abeliana y son de particular importancia en teor\'ia de representaciones.\\

\subsection{Ejemplos}
Presentaremos diferentes ejemplos que permitan ilustrar las ideas ya estudiadas de los $\mathcal{D}$-módulos. Para mayor profundidad v\'ease \cite{B}, \cite{DrM} y \cite{RTT}.

\subsubsection{Ejemplo: El trivial} $\mc{D}_X$ es claramente un $\mc{D}_X$-módulo a izquierda y a derecha.\\

\subsubsection{Ejemplo: Un poco menos trivial}\label{ox} El haz de funciones regulares sobre $X$, $\mc{O}_X$, es un $\mc{D}_X$-módulo a izquierda. La acción de las funciones está dada por multiplicación y las derivadas actúan derivando funciones regulares. M\'as aún, como $X$ es una variedad suave, cuenta con un sistema de coordenadas $\{X_i, \partial_i\}_{i=1}^n$, donde $n=\dim X$. Entonces $\mc{O}_X \cong \mc{D}_X / \mc{D}_X(\partial_1, \ldots, \partial_n)$, donde $\mc{D}_X(\partial_1, \ldots, \partial_n)$ es el ideal generado por $(\partial_1, \ldots, \partial_n)$.\\

\subsubsection{Ejemplo: La ``función'' delta}\label{delta} Similar al ejemplo anterior, al considerar el ideal generado por $X_1, \ldots, X_n$, donde estas son funciones regulares que generan un sistema de coordenadas para $X$, def\'inase el módulo $\Delta_X$ como el cociente $\mc{D}_X/\mc{D}_X(X_1, \ldots, X_n)$. Si se denota la imagen de $1$ en el cociente como $\delta$, se sigue que este elemento satisface $\delta\cdot X_i=0$ para toda $i=1,\ldots, n$.\\

 Defina $\delta$ en un punto $x\in X$ de la siguiente forma. Considere $\Delta_x := \mb{C}_x\otimes_{\mc{O}_X}\mc{D}_X$, este módulo es generado por el símbolo $\delta_x$ y cumple que $\delta_x\cdot f = f(x)$ para $f$ una sección local de $X$.\\

\subsubsection{Ejemplo: Formas diferenciales de orden superior - El $\mc{D}$-módulo a derecha} Sea $\omega_X:= \bigwedge^n \Omega_X$, donde $n=\dim X$. $\omega_X$ es el ejemplo protot\'ipico de $\mc{D}_X$-módulo a derecha. Las acciones est\'an dadas como sigue. Sean $f\in \mc{O}_X$, $\zeta\in \mc{T}_X$ y $\alpha\in\omega_X$, entonces $\alpha\cdot f := f\alpha$ y $\alpha\cdot\zeta := -Lie_{\zeta}(\alpha)$, donde la anterior expresión denota la derivada de Lie. \\

 N\'otese que en efecto la estructura est\'a bien definida. Sean $\zeta, \zeta_1,\zeta_2\in \mc{T}_X$, $f, f_1,f_2\in\mc{O}_X$ y $\alpha\in\omega_X$, 

\begin{itemize}
    \item $(\alpha\cdot f_1)\cdot f_2 = (f_1\alpha)\cdot f_2 = f_2(f_1\alpha) = (f_2f_1)\alpha = (f_1f_2)\alpha = \alpha\cdot (f_1f_2)$.
    \item $(\alpha\cdot f)\cdot \zeta = -Lie_{\zeta}(f\alpha) = -\zeta(f)\alpha - f Lie_{\zeta}(\alpha) = (\alpha\cdot \zeta)\cdot f - \alpha\cdot \zeta(f)$
    \item $\alpha\cdot [\zeta_1,\zeta_2] = -Lie_{[\zeta_1, \zeta_2]}(\alpha) = -( Lie_{\zeta_1}(Lie_{\zeta_2}(\alpha)) - Lie_{\zeta_2}(Lie_{\zeta_1}(\alpha))) = (\alpha\cdot\zeta_1)\cdot\zeta_2 - (\alpha\cdot\zeta_2)\cdot\zeta_1  $
    \item $(\alpha\cdot f)\cdot \zeta =  -Lie_{\zeta}(f\alpha) = -Lie_{f\zeta}(\alpha) = \alpha\cdot(f\zeta) = \alpha\cdot(f\cdot \zeta)$
\end{itemize}

 Las anteriores relaciones se siguen por las propiedades de la derivada de Lie. En general un módulo a izquierda se puede considerar como un m\'odulo a derecha tensorizando por el módulo $\omega_X$.\\

\subsubsection{Ejemplo: Ecuaciones diferenciales}\label{edps} Sea $X$ una variedad algebraica (analítica) suave de dimensión $n$ y considere un operador diferencial $P$. Localmente, para un abierto $U$ de $X$ el operador $P$ se escribe como $P=\sum_{I} g_{I}\partial^{I}$, donde $g_{I}\in \mc{O}_{X}(U)$, $ I=(i_1, \ldots, i_n)\in \mb{N}^{n}$ es un multi-indice y $\partial^{I}=\partial_1^{i_1}\cdots\partial_n^{i_n}$. Se quiere solucionar la ecuación $P(f)=0$ para $f$ en algún $\mc{D}_X$-módulo $\mc{F}$. Note que la naturaleza de las soluciones que se quieran obtener apuntan a que haz $\mc{F}$ escoger, por ejemplo, si queremos soluciones polinomiales el haz a escoger es $\mc{F}=\mc{O}_X$, para el caso de soluciones analíticas tomamos el haz de funciones holomorfas.\\

 Veamos como encontrar el conjunto de soluciones de nuestro sistema de ecuaciones diferenciales. Consideremos un operador matricial $P = [P_{ij}]$ donde $i=1, \ldots p$, $j=1, \ldots, q$ y los $P_{ij}$ son operadores diferenciales. Por tanto, las soluciones al sistema $Pf=0$ son de la forma $f =(f_1, \ldots, f_q)$. Para encontrar estas soluciones del sistema de  ecuaciones $Pf=0$, considere el morfismo de $\mc{D}_X$-m\'odulos $\cdot P : \mc{D}_X^p \to \mc{D}_X^q$ definido como $\cdot P(u) = uP$. Considere ahora la sucesi\'on exacta

\[ 
\xymatrix{ \mc{D}_X^p \ar[r]^{\cdot P} & \mc{D}_X^q \ar[r] & M_P \ar[r] & 0}
\]
\vspace{0.2cm}

 donde $M_P$ denota el cokernel del morfismo $\cdot P$. Al aplicar el functor $\Hom_{\mc{D}_X}(\_ , \mc{F})$, a la anterior sucesi\'on exacta se obtiene

\[ 
\xymatrix{0 \ar[r] & \Hom_{\mc{D}_X}( M_P ,\mc{F}) \ar[r] & \Hom_{\mc{D}_X}(\mc{D}_X^q, \mc{F}) \ar[r]^{P\cdot} & \Hom_{\mc{D}_X}(\mc{D}_X^p ,\mc{F}) }
\]
\vspace{0.2cm}

 Note que $ \Hom_{\mc{D}_X}(\mc{D}_X^r, \mc{F})\cong \mc{F}^r$ mediante el morfismo $(\varphi:\mc{D}_X^r\to \mc{F})\mapsto \{\varphi(e_i)\}_{i=1}^r$ para $\{e_i\}_{i=1}^r$ la base can\'onica $\mc{D}_X^r$ como $\mc{D}_X$-m\'odulo. Por tanto, la anterior sucesi\'on se transforma en  

\[ 
\xymatrix{0 \ar[r] & \Hom_{\mc{D}_X}(M_P ,\mc{F}) \ar[r] & \mc{F}^q \ar[r]^{P\cdot} & \mc{F}^p }
\]

 luego $\Hom_{\mc{D}_X}( M_P ,\mc{F}) = \{ f\in \mc{F}^q \; | \; Pf=0  \} $
\vspace{0.2cm}

 En decir, el espacio $\Hom_{\mc{D}_X}( M_P ,\mc{F})$ contiene todas las soluciones del sistema que pertenecen al $\mc{D}_X$-m\'odulo $\mc{F}$. Se dice que el m\'odulo $M_P$ representa al sistema de ecuaciones diferenciales.

\subsubsection{Ejemplo: La exponencial}\label{exp} Sea $X=\mb{C}$ y sea $\lambda\in \mb{C}$. Defina $E_{\lambda} = \mc{D}_X/\mc{D}_X(\partial-\lambda)$. Este módulo es generado por la imagen del $1$, (el generador de $\mc{D}_X$ como $\mc{D}_X$-módulo), cuya imagen se denota en $E_{\lambda}$ por el símbolo $e^{\lambda z}$. Es decir, $E_{\lambda} \cong \mc{D}_{X} e^{\lambda z}$, y $(\partial - \lambda)e^{\lambda z} = 0$. Nótese que $\partial e^{\lambda z} = \lambda e^{\lambda z}$.\\

 Ahora, como se vio en el Ejemplo \ref{edps}, el espacio $\Hom_{\mc{D}_X}(E_{\lambda}, \mc{F})$ debe capturar la información de las soluciones del operador $P=\partial - \lambda$, que pertenecen al haz $\mc{F}$. Si se toma $\mc{F}=\mc{O}_X$, una solución $f$ que sea global debe cumplir que $f\in \mc{O}_X(X)=\mb{C}[z]$, $(\partial-\lambda)f=0$, la cual no existe. Considerando a $X$ con la topología usual y tomando $\mc{F}$ como el haz de funciones holomorfas, se encontrar\'ia que existe una sección global $f$ tal que $(\partial-\lambda)f=0$; la función exponencial usual.

\subsubsection{Ejemplo: La ``función'' $z^\lambda$}\label{xlambda} Sea $X=\mb{C}$ y sea $\lambda\in \mb{C}$. Considere el módulo $M_{\lambda} = \mc{D}_X/\mc{D}_X(z\partial-\lambda) $. Como en el ejemplo anterior \ref{exp}, a la imagen del $1$ en $\mc{D}_X$ se denotará por $z^{\lambda}$. Así se tiene que $M_{\lambda} \cong \mc{D}_X z^{\lambda}$. Además se cumple la ecuación diferencial $(z\partial-\lambda)z^{\lambda}=0$, i.e., $z\partial z^{\lambda} = \lambda z^{\lambda}$.\\

 Adicionalmente, sobre $M_{\lambda}$ se tiene lo siguiente. Defina el operador $E=z\partial$ en $\mc{D}_X$ y note que $[E,z]=z$ y $[E, \partial]=-\partial$. Por lo tanto, se tiene que $Ez=z(E+1)$ y $E\partial = \partial(E-1)$. Luego en el módulo $M_{\lambda}$, $z$ y $\partial$ son vectores propios del operador $E$ con valores propios $\lambda+1$ y $\lambda -1$ respectivamente. Es decir, se ha obtenido que $Ez=z(\lambda+1)$ y $E\partial = \partial(\lambda-1)$. Por inducción en $n$ se sigue que $Ez^n=z^n(\lambda+n)$ y $E\partial^n = \partial^n(\lambda-n)$.\\

 Defina las siguientes acciones $z\cdot z^{\lambda} = z^{\lambda+1}$, $\partial\cdot z^{\lambda} = \lambda z^{\lambda -1}$ y $E\cdot z^{\lambda} = (\lambda)z^{\lambda}$. Si $\lambda\notin \mb{Z}$, ninguna de las anteriores acciones se anula, por lo tanto, el módulo $M_{\lambda}$ es irreducible. Caso contrario si $\lambda\in \mb{Z}$, existir\'a una potencia de $z$ que se anula y por tanto la acción de $\partial$ en este elemento se anula, así el módulo será reducible.\\

 Los ejemplos \ref{ox}, \ref{delta}, \ref{edps}, \ref{exp} y \ref{xlambda} son, adem\'as, ejemplos de módulos holonómicos.

\section{Aplicaciones}

 En esta secci\'on, se presentan dos de las aplicaciones m\'as c\'elebres de la teoría de los $\mc{D}$-m\'odulos. La correspondencia de Riemman-Hilbert y el teorema de localizaci\'on de Beilinson y Bernstein.\\

 La correspondencia de Riemann-Hilbert, en su forma original se pregunta sobre la existencia de sistemas de ecuaciones diferenciales regulares, dada una representación del grupo fundamental de la variedad en la cual está definido el sistema. En t\'erminos m\'as generales, la correspondencia de Riemann-Hilbert afirma que conexiones singulares regulares sobre una variedad fija están en correspondencia uno a uno con representaciones de su grupo fundamental. En el libro \cite{D}, P. Deligne demuestra este teorema en el contexto de haces y categorías derivadas. Una introducci\'on a este tema se puede encontrar en la primera parte del libro \cite{RTT} y en los cap\'itulos 3 y 4 del libro \cite{Bo}.\\

 La segunda aplicaci\'on, el teorema de localizaci\'on de Beilinson y Bernstein, afirma que toda representaci\'on de un \'algebra de Lie semisimple se puede entender geom\'etricamente como un $\mc{D}$-m\'odulo sobre la variedad bandera del \'algebra de Lie. Esta importante relaci\'on permiti\'o demostrar las conjeturas de Kazhdan-Lusztig las cuales permiten identificar todos los m\'odulos simples que componen a un m\'odulo de Verma, (un m\'odulo indescomponible de peso m\'aximo), es decir, permiten conocer el caracter de cualquier representaci\'on irreducible en t\'erminos de los m\'odulos de Verma. Además de esto, este teorema dió inicio a la teor\'ia geom\'etrica de representaciones. La demostración de este teorema está publicada en el art\'iculo \cite{BB}. Una introducci\'on compacta al tema se encuentra en las notas de D. Gaitsgory ``{\it Geometric Representation Theory}'', \cite{Ga}. En la segunda parte del libro \cite{RTT} se encuentra el tema de manera comprensiva además de la demostraci\'on de las conjeturas de Kazhdan - Lusztig. En los art\'iculos \cite{KL} y \cite{KL2} se puede encontrar el planteamiento original de las conjeturas de Kazhdan - Lusztig y su estudio utilizando geometr\'ia algebraica. \\ 

\subsection{La correspondencia de Riemman - Hilbert} Sea $X$ una variedad algebraica suave. En la sección \ref{Dmoddef} se defini\'o una conexión como un $\mc{D}_X$ módulo que es $\mc{O}_X$ coherente. En particular, éstos módulos son  $\mc{D}_X$-módulos localmente libres de rango finito sobre $\mc{O}_X$. La categoría de las conexiones sobre $X$ se denota por $Conn(X)$.\\

 Sea $\mc{F}\in Conn(X)$. Es claro que la estructura de $\mc{D}_X$-módulo de $\mc{F}$ es equivalente a la existencia de un morfismo de álgebras de Lie $\nabla : \mc{T}_X \rightarrow End_{\mb{C}}(\mc{F})$. Por lo tanto, tiene sentido considerar secciones $s$ tales que $\nabla_{\theta}(s)=0$ para toda sección del haz tangente $\theta$. Considere el haz $\mc{L}_{\mc{F}}$ definido por las secciones para las cuales la conexión se anula, es decir, para $U\subseteq X$ abierto, sea $\mc{L}_{\mc{F}}(U) = \{ s\in \mc{F}(U) \;|\; \nabla_{\theta}(s) = 0\; \forall \theta\in \mc{T}_X \} $. Se puede mostrar que $\mc{L}_{\mc{F}}$ es un haz localmente isomorfo a $\mb{C}$, es decir, es un haz de $\mb{C}$-espacios vectoriales. Este tipo de haces se llaman \emph{sistemas locales}. Recíprocamente, dado un sistema local $\mc{L}$, se puede construir una conexión como $\mc{F}_{\mc{L}} := \mc{O}_X\otimes_{\mb{C}}\mc{L}$.\\

 Con lo anterior se puede escribir una primera instancia de la correspondencia de Riemman - Hilbert:

\begin{teo}[\cite{RTT}, Corolario 5.3.10]\label{teo2}
La categoría de conexiones sobre $X$ es equivalente a la categoría de sistemas locales sobre $X$.
\end{teo}

 Existe una versión mucho más sofisticada del Teorema \ref{teo2} el cual va m\'as all\'a de los propósitos de estas notas. No sobra decir que este Teorema es un primer resultado en el estudio de sistemas de ecuaciones diferenciales en dominios del plano complejo. La idea es que dada una ecuación diferencial es posible asociarle una matriz (llamada \emph{monodromía}) que permite medir la diferencia entre las soluciones locales y globales del sistema. Adicionalmente, esta matriz corresponde con a una representación del grupo fundamental del dominio sobre el cual est\'a definido el sistema en cuestión. Esta representación a su vez se puede interpretar como un sistema local en el dominio de interés. La pregunta original, el problema número 21 de Hilbert, es el siguiente. Dada una matriz de este tipo (i.e. de monodromía) es siempre posible encontrar una ecuación diferencial que realice tal matriz. La respuesta fue afirmativa gracias al Teorema \ref{teo2}. 

\subsection{El teorema de Localización de Beilinson y Bernstein} Como se vió en el Ejemplo \ref{locBBP1}, $\phi:U(\mf{sl}_2(\mb{C}))/ \langle C \rangle \rightarrow \Gamma(\mb{P}^1, \mc{D}_{\mb{P}^1})$ es un isomorfismo de álgebras, donde $C$ es el elemento de Casimir. Se puede mostrar que $C$ genera el centro $Z=Z(U(\mf{sl}_2(\mb{C}))$ del álgebra envolvente $U(\mf{sl}_2(\mb{C}))$, por lo tanto el isomorfismo $\phi$ se escribe as\'i $U(\mf{sl}_2(\mb{C}))/ Z\cdot U(\mf{sl}_2(\mb{C})) \rightarrow \Gamma(\mb{P}^1, \mc{D}_{\mb{P}^1})$. Este isomorfismo se puede extender a una versi\'on m\'as general v\'alido para cualquier álgebra de Lie semisimple.\\

 Sea $\g$ un álgebra de Lie semisimple y sea $U(\g)$ el álgebra envolvente de $\g$. Se denota por $Z$ el centro de $U(\g)$. Sea $X$ la variedad bandera de $\g$, es decir, la variedad que tiene como elementos todas las subálgebras de Borel de $\g$. En particular para $\g=\mf{sl}_n(\mb{C})$, $X = \{F = (0 = V_0 \subset V_1 \subset \cdots \subset V_{n-1} \subset V_{n} = \mb{C}^n )\; | \; \dim_{\mb{C}}V_i = i\}$. Así, para $\mf{sl}_2(C)$, $X$ es el espacio de líneas en $\mb{C}^2$, es decir, $X=\mb{P}^1$.\\

 Sea $\mc{D}_X$ el haz de operadores diferenciales en la variedad bandera $X$ y $\Gamma(X, \mc{D}_X)$ el conjunto de secciones globales.

\begin{teo}[\cite{Ga}, Teorema 6.18]\label{teo1} 
 Existe un isomorfismo de álgebras $U(\g)/Z\cdot U(\g) \rightarrow \Gamma(X, \mc{D}_X)$.
\end{teo}

 M\'as aún, se puede demostrar que el functor $\Gamma(X, \;)$ no tiene cohomología en grado superior y que todo $\mc{D}_X$-módulo $\mc{F}$ está generado por las secciones globales. A partir de estos resultados junto con el Teorema \ref{teo1} se puede demostrar el siguiente:

\begin{teo}[Beilinson - Bernstein, \cite{BB}]
La categoría de $U(\g)$-m\'odulos con acción trivial del centro (i.e. la categoría de $U(\g)/Z\cdot U(\g)$-módulos) es equivalente a la categoría de $\mc{D}_X$-módulos que son $\mc{O}_X$-cuasi-coherentes.
\end{teo}

 \begin{proof} (Idea de la demostraci\'on.) El functor $\Gamma(X, \_)$ tiene por adjunto a izquierda el functor $Loc(\quad) = \mc{D}_X\otimes_{U(\g)/Z\cdot U(\g)}(\quad)$. Por lo tanto tenemos el morfismo de adjunci\'on $1_{U(\g)/Z\cdot U(\g)-\md}\Rightarrow  \Gamma(X, \quad) \circ Loc$. Ahora, para $I, J$ dos conjuntos, considere $N\in U(\g)/Z\cdot U(\g)-\md$ y la secuencia exacta $(U(\g)/Z\cdot U(\g))^{I}\rightarrow(U(\g)/Z\cdot U(\g))^{J}\rightarrow N \rightarrow 0$. Entonces, obtenemos el siguiente diagrama conmutativo

\[\xymatrix{ (U(\g)/Z\cdot U(\g))^{I} \ar[r] \ar[d] & (U(\g)/Z\cdot U(\g))^{J} \ar[r] \ar[d] &  N \ar[r] \ar[d] & 0\\
\Gamma(X, Loc((U(\g)/Z\cdot U(\g))^{I}))  \ar[r] & \Gamma(X, Loc((U(\g)/Z\cdot U(\g))^{J}))  \ar[r] & \Gamma(X, Loc(N))  \ar[r] & 0.}  \]
\vspace{0.1cm}

 Por el teorema \ref{teo1}, $\Gamma(X, Loc((U(\g)/Z\cdot U(\g))^{K}))\cong \Gamma(X, (\mc{D}_X)^{K})) \cong (U(\g)/Z\cdot U(\g))^{K}$  para $K=I,J$, as\'i $\Gamma(X, Loc(N))\cong N$. An\'alogamente, para $\mc{M}$ un $\mc{D}_{X}$-m\'odulo se obtiene $Loc(\Gamma(X,\mc{M}))\cong \mc{M}$. 

\end{proof}



\end{document}